\documentclass[10pt]{article}
\usepackage{graphicx}%
\usepackage{multirow}%
\usepackage{amsmath,amssymb,amsfonts}%
\usepackage{amsthm}%
\usepackage{mathrsfs}%
\usepackage[title]{appendix}%
\usepackage{xcolor}%
\usepackage{textcomp}%
\usepackage{manyfoot}%
\usepackage{booktabs}%
\usepackage{url}
\usepackage{natbib}
\usepackage{algorithm}
\usepackage{algorithmic}
\usepackage{listings}%

\usepackage[utf8]{inputenc}
\usepackage{blkarray}
\usepackage{setspace}
\usepackage{enumitem}
\everymath{\displaystyle}

\usepackage[dvipsnames]{xcolor}
\usepackage{hyperref}
\definecolor{darkmintblue}{HTML}{2E7F7F}

\hypersetup{colorlinks,
	citecolor=darkmintblue,
	citebordercolor=darkmintblue,
	linkcolor=darkmintblue,
	urlcolor=MidnightBlue
}

\usepackage{graphics} % for pdf, bitmapped graphics files
\usepackage{epsfig}   % for postscript graphics files
\usepackage[dvipsnames]{xcolor}

\usepackage{url}
\usepackage{dsfont}
\usepackage{mathtools}
\usepackage[normalem]{ulem}

\newcommand{\EE}{\mathbb{E}}

\newcommand{\ubar}[1]{\underline{#1}}
\newcommand{\ubarc}{\underline{c}}

\usepackage{blkarray}

\newcommand\cC{\mathcal C}
\newcommand\cD{\mathcal D}

\newcommand\bu{\boldsymbol{u}}
\newcommand\bp{\boldsymbol{p}}

\newcommand\bZ{\boldsymbol{Z}}

\newcommand\rmS{\mathrm{S}}
\newcommand\rmI{\mathrm{I}}
\newcommand\rmR{\mathrm{R}}

\newtheorem{theorem}{Theorem}[section]

\newtheorem{define}[theorem]{Definition}
\newtheorem{lemma}[theorem]{Lemma}

\newtheorem{assumption}[theorem]{Assumption}

\def\bZ{\boldsymbol{Z}}

\allowdisplaybreaks

\usepackage{subcaption}
\usepackage{enumitem}
\setlist[itemize]{noitemsep, topsep=2pt, parsep=2pt, partopsep=2pt}

\title{Modeling Epidemic Spread with Strategic Vaccination and Socialization: a Mean Field Game Analysis\thanks{This work is supported by National Science Foundation under grant DMS-2436332.}}
\author{
Huaning Liu\thanks{Department of Statistics, University of Illinois Urbana-Champaign, Champaign, IL 61820, USA.}
\ \ \ \ \ \
G\"{o}k\c{c}e Dayan{\i}kl{\i}\footnotemark[2] \ \thanks{Corresponding author. Email address: {\tt gokced@illinois.edu}} 
}
\date{}

\begin{document}

\maketitle

\begin{abstract}
We study a game-theoretic model of epidemic control in a large population with finitely many groups and non-cooperative individuals. In the model, individuals jointly choose their socialization levels and vaccination rates, and vaccination is subject to a linear individual cost structure. We derive a forward-backward ordinary differential equations (FBODE) system that characterizes the mean field Nash equilibrium, show that the equilibrium vaccination rate exhibits an at-most one-jump bang-bang structure, and establish the existence of a Carath\'eodory solution to the FBODE. This establishes a realistic interpretation of the vaccination decisions, meaning individuals decide to vaccinate until a time point which is determined by model parameters and then stop after. We further consider a population-awareness extension in which individuals incorporate population infection information into their objective functions, and we prove a similar at-most one-jump bang-bang property under suitable conditions. Finally, we propose a numerical algorithm for solving the FBODE and conduct simulations to validate the theoretical findings. The experiments highlight two main insights: the trade-off between socialization and vaccination, and the greater importance of quarantining infected individuals instead of restricting susceptible individuals.
\end{abstract}

\section{Introduction}
During the COVID-19 pandemic, it became evident that individuals may not fully comply with the epidemic mitigation policies such as social distancing or mask mandates. Instead, some individuals may engage in additional precautionary behaviors, for example they may wear masks outdoors when it is not mandated~\cite{borter_2021,kim_shin_2022}, while others may choose not to follow these mandates~\cite{knutson_2020,aguilera_2020}. These observations highlight that the effectiveness of epidemic mitigation policies depends not only on their biological impact but also on the behavioral responses they elicit. Consequently, designing robust mitigation strategies requires incorporating models of how individuals adjust their behaviors in response to perceived risks, costs, and policy incentives. However, capturing these behavioral responses is a challenging task, because it requires modeling interactions among many individuals who each aim to optimize personal objectives.

The Susceptible-Infected-Recovered (SIR) framework provides a powerful mathematical tool for modeling the spread of an infectious disease at the population level. Classical epidemic-control literature built on this framework interjects the effects of public-health policies (social-distancing guidelines) exogenously in the model as effects on the parameters of the SIR dynamics. In this way, they naturally formulate an optimal control problem in which a public health authority selects the mitigation policies to optimize their own objectives such as the overall societal welfare (e.g.~\cite{Bolzoni_2017,Rossa_2024,Bolzoni_2025}). However, evidence suggests that even policies do shift people's behavior towards its goal (e.g., decreasing social interaction to prevent infections), a nontrivial fraction of individuals exhibit imperfect compliance for these policies~\cite{Murray_2021,Levy_2022}. This gap motivates constructing mathematical models that explicitly incorporate individual-level decision-making regarding the prevention measures and human behavioral response to mitigation policies, where individuals rationally adjust their behaviors to maximize personal objectives under perceived risks and costs. This task includes modeling of a large number of interacting individuals; therefore, game theoretical tools need to be employed in the mathematical models. 

In our model, we develop a continuous-time multi-population mean field game model in which individuals choose simultaneously their socialization and vaccination levels over a continuous time horizon to minimize their individual costs related to their non-compliance, treatment, and vaccination. We focus on finding the Nash equilibrium behavioral trajectories under exogenous mitigation policies. This means that none of the individuals improve their objective functional by changing their behavioral trajectory over the time horizon. With this model, we aim to understand the emergent behavioral responses and the possible trade-offs between different individual prevention measures (i.e., social distancing vs vaccination). Due to the linear vaccination cost formulation, at the equilibrium we observe that individuals vaccinate until a time point and stop vaccinating after this time, this incorporates a more realistic interpretation for equilibrium vaccination decisions. Furthermore, we incorporate \textit{population-awareness} into the individual cost functional through the proportion of infected individuals, which allows us to study its effect on individual behavior. We prove that with population-awareness, the individuals choose to vaccinate longer under the same epidemic conditions. Our numerical results reveal a trade-off relation between vaccination and socialization, with the balance between the two determined by the population-awareness parameters. They further highlight the importance of suggesting stricter distancing guidelines for infected individuals during the early stage of an epidemic outbreak.

\subsection{Literature Review}
\noindent{\textbf{Mathematical Modeling of Epidemic Control.}}
The modeling of epidemic control has long enjoyed interest from epidemiologists and the control community. Several earlier studies have investigated optimal control formulations for disease outbreaks within the classical SIR epidemic context, where interventions enter through vaccination (affecting the transition from S to R), isolation (affecting the transition from I to R), or contact reduction (scaling the transmission rate)~\cite{Behncke_2000_deter_epi}. In these works, a social planner typically seeks to minimize a weighted sum of infection prevalence and control effort, which, under linear control costs, often leads to bang-bang type optimal policies~\cite{Wickwire_1975_optisolation, Bolzoni2014_wildlife}. More recent work extends this line of research by incorporating additional realism, including control budget constraint~\cite{Hansen2011_limited}, time-optimal objectives that minimize the duration required for infection levels to fall below a prescribed threshold~\cite{Bolzoni_2017}, hard healthcare capacity constraints (imposed as bounds on infection prevalence)~\cite{Rossa_2024}, and Erlang-distributed infectious periods~\cite{Bolzoni_2025}. These extensions give rise to control problems with qualitatively different optimal strategies. As mentioned earlier, these models do not incorporate decision-making process of the individuals, instead focus on modeling the policy-maker and her optimal intervention policies.

From a broader game‑theoretic modeling standpoint, in~\cite{Fine_1986}, authors constructed a game theoretical modeling for vaccination and concluded rational individuals choose a lower vaccination uptake than the socially optimal coverage in the presence of herd-immunity. A similar free-rider problem for infectious disease was also discussed later in~\cite{brito_1991}. In~\cite{Bauch_2004}, authors studied a stationary vaccination game: individuals choose the probability of vaccinating (i.e. a mixed strategy on vaccinate or not) and minimize the expected risks, where the infection probability is informed by a compartment model and used to explicitly derive the equilibrium strategy. The uniqueness of Nash equilibrium in vaccination games is independently discussed in~\cite{Bai_uniqueness}. 

The social distancing is first studied in the context of general game theory in~\cite{reluga_sd_2010}, which formulates an SIR population game where individuals choose social distancing investment that collectively affects the transition rate from S to I, through a decreasing efficiency function with diminishing marginal effectiveness. This framework is extended in~\cite{Reluga_2013_plinear} by adopting a piecewise-linear efficiency function, and in~\cite{li_lindberg_smith_reluga} by allowing the function to depend jointly on personal and public policy investments. In contrast, we do not postulate such an efficiency map; instead, individuals' transition rate from S to I is determined directly by their own socialization level (as the control) jointly with the population’s socialization profile, yielding a more individualistic formulation. We also refer to~\cite{optimality_doncel_2025} as a more recent advance that models social distancing in a finite-agent discrete-time stochastic game
\vskip2mm

\noindent{\textbf{Mean Field Games and Extensions.}}
Mean field games (MFGs) are introduced by
upon the seminal contribution of 
Lasry and Lions~\cite{Lasry_2007} and Huang, Malham{\'e} and Caines~\cite{Huang_2006} to approximate the Nash equilibrium in the large population dynamic games. MFGs enable mean field approximations as the population size approaches infinity. This helps tackle the curse of dimensionality problem in large population games that comes from the increasing number of pair interactions. By assuming identical and infinitesimal agents, MFGs reduce the Nash equilibrium characterization to the analysis of a coupled pair of (partial or stochastic) differential equations. These equations intuitively describe the representative agent’s optimization and the evolution of the population distribution. MFGs have been applied broadly across domains, including energy markets~\cite{aid2020entry, djehiche2020price, carmona2022mean, alasseur2020extended, bagagiolo2014mean,bichuch_acc}, traffic management~\cite{chevalier2015micro, huang2021dynamic, festa2018mean}, systemic risk~\cite{carmonafouquesun2015mean}, and cyber security~\cite{Kolokoltsov2016cyber}. Multi-population MFGs were first mentioned in~\cite{Lasry_2007} as an important direction for future research. Since then, they have been extensively developed and applied in various modeling contexts, we refer to~\cite{Feleqi2013, mfg_mfc_multi, multi_major_minor} as examples.

In particular for vaccination decisions, Laguzet and Turinici~\cite{Laguzet_2015} first discussed a dynamic vaccination game under the context of MFGs, where a \textit{vaccinated} state is introduced and individuals choose whether and when to get vaccinated according to their own objectives. The agents' strategy space is the time switching to vaccinated and the mean field interaction lies on the cumulative density of individuals who have get vaccinated, from which the existence of Nash equilibrium is established. Doncel, Gast, and Gaujal~\cite{Doncel_Gast_Gaujal_2022} independently studied vaccination modeling in SIR via both MFG and mean field control (MFC) where a social planner is prescribing the \textit{optimal vaccination} strategy to the individuals. They show that when individuals can directly choose a vaccination rate (i.e. the transition rate from susceptible to immune), the equilibrium control is of bang–bang form, and individuals vaccinate for a shorter period under competition than in the MFC setting. Indeed the admissible vaccination strategy spaces in~\cite{Laguzet_2015} and~\cite{Doncel_Gast_Gaujal_2022} are formulated differently. The former implicitly restricts admissible controls to those exhibiting exactly one jump, while the latter considers a compact vaccination strategy space and shows that the equilibrium control is of bang–bang type. We adapt the latter formulation in this paper.

On the other hand, for social distancing, individual-level decision making has been introduced and analyzed with mean field games~\cite{elie2020contact,pnas_epidemics,liu2025incorporatingauthorityperceptioneconomic,buckley2025behavioralpatternsmeanfieldgames}, and their extensions to Stackelberg mean field games~\cite{aurell2022optimal} to include a public-health authority and graphon games~\cite{aurell2022finite} to include network structures. 
For additional epidemic models formulated within the frameworks of MFGs and MFC, we refer the readers to~\cite{bremaud2024mean,cho2020mean}. Different than the literature that is presented, our model incorporates both the vaccination and social distancing controls in the model to analyze possible trade-offs between them. Furthermore, due to the linear vaccination cost construction our model will have an equilibrium vaccination level that is piecewise constant with one jump (i.e., bang-bang type).
\vskip2mm

\noindent{\textbf{Switched Systems and Control.}}
Hybrid systems combine continuous dynamics with discrete events. Among them, switched system is a particular class of models consisting of a family of continuous-time subsystems and a rule orchestrating the switch between them. Many real-world processes, such as automotive systems, mechanical systems, and traffic-control systems, can be modeled in this way, and the subject has therefore attracted considerable attention. A rich literature has been developed for switched systems, particularly on stability and stabilization via common or multiple Lyapunov functions, as well as on controllability, observability, and optimal control synthesis~\cite{Liberzon2003, liberzon1999basic, sun2005analysis, shorten2007stability}. Game-theoretic extensions have also been studied, including zero-sum hybrid differential games characterized through quasi-variational inequalities and mean-field models with switching/stopping decisions~\cite{Dharmatti2006zerosum,FIROOZI2022aclass,fabio2023atime,Bensoussan2020meanfieldjump}. Our work is related to this line of research because the equilibrium vaccination rate exhibits a bang–bang control structure reminiscent of switching behavior. Unlike classical switched-control formulations where the switching law is prescribed or directly optimized, here the switching pattern emerges endogenously from equilibrium best responses in an epidemic mean field game with coupled vaccination and socialization decisions. See, for example,~\cite{Lin2013optimalnonlinear} for a discussion of the distinction between bang-bang control structure arising in single-mode systems and genuinely discrete-valued controls from an optimal control literature perspective.

\subsection{Contributions and Paper Structure}
Our contributions are threefold. First, from a modeling perspective, we study a multi-population epidemic MFG in which individuals jointly choose socialization level and vaccination rate. Extending the model setup in~\cite{aurell2022finite,liu2025incorporatingauthorityperceptioneconomic}, we adopt a more realistic linear vaccination cost and study its implications for equilibrium behavior. We then propose a \textit{population-awareness} extension of this baseline formulation, in which population-level infection information is endogenized into individual objectives, allowing us to quantify how awareness feeds back into individual decisions and collective outcomes. Second, on the theoretical side, we prove that equilibrium vaccination level admits an at-most one-jump piecewise-constant structure, extending the result of~\cite{Doncel_Gast_Gaujal_2022}. We further provide a full equilibrium characterization via a differential equations system and establish its existence result for the baseline model. Due to the introduction of social-distancing decisions, the theoretical analysis of the inducing extended MFG becomes much more complex than that in~\cite{Doncel_Gast_Gaujal_2022}, and requires new analytical tools as developed in Theorem~\ref{thm:fbode_wellposedness}. For the population-awareness extension, we provide a parameter sufficient condition under which the at-most one-jump structure is preserved. Third, we validate theory numerically by developing an algorithm to solve the differential equations system, and we conduct sensitivity analysis with the model parameter values calibrated from survey real data and validate our theoretical results.

The paper structure is as follows. In Section~\ref{sec:mfg_res}, we introduce the multi-population epidemic MFG model, mean field Nash equilibrium definition, and the theoretical results related to the mean field Nash equilibrium. In Section~\ref{sec:pop-aware}, we adjust our proposed model to include the population-awareness in the objectives of the individuals and compare this model with the base one. In Section~\ref{sec:numerics}, we introduce our numerical approach and present the experiment results.

\section{Mean Field Game Model for Epidemic Control}
\label{sec:mfg_res}
\subsection{Preliminaries}
We denote $T > 0$ as a fixed finite time horizon. Let $\cC:=C([0, T] ; \mathbb{R})$ be the space of continuous  real-valued functions on $[0,T]$, and let $\cD:=D([0, T] ; \mathbb{R})$ be the space of real-valued functions that are c\`adl\`ag on $t \in [0,T)$ and continuous on $t = T$. We also denote $AC([0,T];\mathbb{R})$ as the space of absolutely continuous real-valued functions on $[0, T]$. Consider the finite epidemic state space $E := \{\rmS, \rmI, \rmR\}$, where S, I, and R correspond to the \textit{Susceptible}, \textit{Infected}, and \textit{Recovered} (or \textit{Removed}) compartments respectively. We continue with introducing the multi-population MFG model that is the limit of a finite-agent model where the number of agents (i.e., individuals) go to infinity. The interested reader can find similar motivating finite-agent models in \cite[Section 1.1]{aurell2022optimal}, which focus on social distancing decisions and are among the first to discuss it in the MFG context; or in \cite[Supporting Information A]{liu2025incorporatingauthorityperceptioneconomic}, which further incorporates vaccination decisions with quadratic costs.

\subsection{Mean Field Game Model and Equilibrium Notions}
    We assume that the population has $K < \infty$ many groups, and each group is equipped with different set of model parameters to represent group heterogeneities. As the population size goes infinitely large, the proportion of the group sizes stays the same, namely $m^{k}$ with $\sum_{k=1}^{K} m^k = 1$. In the MFG model, we assume that individuals within each group are homogeneous and interact with other agents via the population distribution of the states and controls. Therefore, we focus on a representative agent from group $k \in [K] := {1,2,\dots,K}$. Without loss of generality, we assume that $m^k > 0$ for all $k \in [K]$. We adopt the classical SIR epidemic framework, in which individuals transition among the three health states over the time interval $[0,T]$. Representative agent $k$ chooses a control tuple consisting of socialization and vaccination decisions. For each state $e \in E$, her socialization control is denoted by $\alpha^{k}(e) = (\alpha_t^{k}(e))_{t \in [0,T]} \in \cC$, where $\alpha_t^{k}(e) \in [0,1]$ for all $t \in [0,T]$, and it is assumed to be Markovian and square-integrable. Her vaccination rate control is denoted by $\nu^{k} = (\nu_t^{k})_{t \in [0,T]} \in \cD$, where $\nu_t^{k} \in [0,V]$ for all $t \in [0,T]$, and it is also assumed to be Markovian and integrable. For simplicity, we normalize the upper bound of the vaccination rate to $V = 1$, as any other value of $V$ can be absorbed by rescaling the corresponding vaccination parameters. A control tuple $(\alpha^k(e), \nu^k)_{e \in E}$ satisfying the above conditions is called \textit{admissible}. Throughout this work, we assume that immunity does not wane, so the recovered state $\rmR$ is absorbing. This simplification is appropriate for the relatively short time horizon considered in this study. For example, in our simulation study informed by SARS-CoV-2 data in Section~\ref{sec:numerics}, we consider an 80-day horizon, whereas prior studies indicate that immune protection could remain detectable for over six months~\cite{waning_immu}.

    We define the probability simplex over $E$ by $\Delta_{E}:=\{p=\left(p_{\rmS}, p_{\rmI}, p_{\rmR}\right) \in \mathbb{R}_{+}^3: \sum_{e \in E} p_e=1\}$ and equip it with Euclidean distance. The initial state distribution of population for group $k$ is set as a \textit{known} vector $(\pi_0^k(e))_{e \in E} \in \Delta_{E}$.
Let $(X_t^k)_{t \in [0,T]}$ be the state process of representative agent $k$, the controlled state dynamics follows a continuous-time Markov chain (CTMC) that evolves on $E$, and its transition-rate matrix writes
\small
\begin{equation*}
Q(\alpha_{t}^{k}, \nu_{t}^{k};Z_{t}^{k}) = 
\begin{blockarray}{ccccc}
& \text{S} & \text{I} & \text{R} \\
\begin{block}{c(cccc)}
  \text{S} & \cdots &
\beta^k \alpha_t^k Z_t^k & \kappa^{k}\nu_{t}^{k} \\
  \text{I} & 0 & \cdots & \gamma^{k} \\
  \text{R} & 0 & 0 & \cdots \\
\end{block}
\end{blockarray}
\end{equation*}
\normalsize
The notation $\cdots$ denotes the negative sum of the other entries in the same row, ensuring that each row sums to $0$. Here $\beta^{k} > 0$ is the base transmission level of group $k$. Similarly $\kappa^{k} > 0$ is a group-specific vaccination efficacy parameter. With higher $\kappa^{k}$, an individual is expected to transition to \textit{Recovered} (R) state in a shorter time after getting vaccinated. When representative agent $k$ gets infected, the time until her recovery (i.e., to state R) is exponentially distributed with a group-specific rate of $\gamma^{k}$. In the model, representative agent $k$ is also impacted by a mean-field interaction term $Z_{t}^{k}$ that weights the socialization levels of infected individuals across all groups, and it is defined $Z_t^k=\sum_{l \in[K]} w(k, l) \bar{\alpha}_t^l(\rmI) p_t^l(\rmI) m^l$.
Here, $w(k,l) \in [0,1]$ is the connection strength between group $k$ and $l$, and it captures the underlying social network structure across groups. The term $p_t^l(\rmI)$ is the proportion of infected individuals in group $l$. Moreover, $\bar{\alpha}_t^l(\text{I}):=\EE[\alpha_t^l(\text{I})]$ is the average socialization level of infectious individuals from group $l$. The cost function of representative agent $k$ is given as
\begin{equation}
\begin{aligned}
J^k\left(\alpha^k, \nu^k ; Z\right)&=\mathbb{E}\bigg[\int _ { 0 } ^ { T } \Big[\big(c_\lambda^k(\lambda_t^{k, \mathrm{S}}-\alpha_t^k)^2+c_\nu^k\nu_t^k\big) \mathds{1}_{\mathrm{S}}(X_t^k) \\[2mm]
& +\big((\lambda_t^{k, \mathrm{I}}-\alpha_t^k)^2+c_{\mathrm{I}}^k\big) \mathds{1}_{\mathrm{I}}(X_t^k)+(\lambda_t^{k, \mathrm{R}}-\alpha_t^k)^2 \mathds{1}_{\mathrm{R}}(X_t^k)\Big] d t\bigg].
\end{aligned}
\end{equation}
Here $\mathds{1}_{e}(X_t^{k})$ denotes the indicator function, that is equal to $1$ when the input current state is $e$, and $0$ otherwise. Compared with the prior work~\cite{liu2025incorporatingauthorityperceptioneconomic}, we adopt a linear vaccination costs in the form of $c_{\nu}^k \nu_t^k$, which better reflects the marginal vaccination incentives. Regarding the parameters, $\lambda_t^{k,e} \in (0,1]$ denotes the social-distancing guideline for individuals in group $k \in [K]$ and health state $e \in E$ that is exogenously set by a public health authority. In this work, we exclude the case of full lockdown, that is, we assume $\lambda_t^{k, e} > 0$ for all $t\in[0,T], k\in[K]$, and $e\in E$. Individuals are penalized if their socialization levels deviate from the socialization level guidelines. This accounts for the penalties for violating the guidelines as well as possible opportunity costs of lowered socialization levels. Upon infection, individuals in group $k$ incur a group-specific, time-independent infection cost $c_{\mathrm{I}}^{k} > 0$. It quantifies both tangible expenses (e.g., medication or treatment costs) and intangible burdens (e.g., the physical and emotional suffering experienced during illness). Lastly, $c_{\lambda}^{k} > 0$ and $c_{\nu}^{k} > 0$ weight the costs of policy deviation and vaccination in the susceptible state respectively. We next define the equilibrium notion.

\begin{define}
    The control profile $(\hat{\alpha}^k, \hat{\nu}^{k}, \hat{Z}^k)_{k \in [K]}$ is called a multi-population mean field Nash equilibrium (MFNE) if for any admissible $(\alpha^k, \nu^k)_{k \in [K]}$, we have
    $$J^{k}(\hat{\alpha}^{k}, \hat{\nu}^{k}; \hat{Z}) \leq J^{k}(\alpha^{k}, \nu^{k}; \hat{Z})$$
    for all $k \in [K]$ where $\hat{Z} = (\hat{Z}_{t}^{k})_{t \in [0, T], k \in [K]}$ and $\hat{Z}_{t}^{k} = \sum_{l \in[K]} w(k, l) \bar{\hat{\alpha}}_t^l(\rmI) p_t^l(\rmI) m^l$.
\end{define}

\subsection{Main Theoretical Results}
In this section, we first provide a full characterization of the MFNE through a forward–backward ordinary differential equations (FBODEs) system. We then analyze the equilibrium vaccination rate at the MFNE and show that it follows an at-most one-jump bang–bang structure. We also establish the existence of the FBODE solutions. To this end, we first define the value function of representative agent $k$ in state $e \in \{\mathrm{S},\mathrm{I},\mathrm{R}\}$ as the minimal expected cost from time $t$ onward, that is
\begin{equation}
    \begin{aligned}
u_t^k(e):=&\inf _{(\alpha_\tau^k, \nu_\tau^k)_{\tau \in[t, T]}} \EE\bigg[\int_ {t}^{T} \Big[\big(c_\lambda^k(\lambda_\tau^{k, \mathrm{S}}-\alpha_\tau^k)^2+c_\nu^k\nu_\tau^k\big) \mathds{1}_{\mathrm{S}}(X_\tau^k) \\
& \hspace{1.5cm} +\big((\lambda_\tau^{k, \mathrm{I}}-\alpha_\tau^k)^2+c_{\mathrm{I}}^k\big) \mathds{1}_{\mathrm{I}}(X_\tau^k)+(\lambda_\tau^{k, \mathrm{R}}-\alpha_\tau^k)^2 \mathds{1}_{\mathrm{R}}(X_\tau^k)\Big] d \tau \mid X_t^k=e\bigg].
\end{aligned}
\end{equation} 
We begin by discussing a regularity property of the value function $u_t^k(e)$, namely its continuity in time, which will be used repeatedly in the theoretical analysis below. For this purpose, we define the running cost
\begin{equation}
    f^k(t,e,\alpha,\nu)
:=
\big(c_\lambda^k(\lambda_t^{k,\rmS}-\alpha)^2+c_\nu^k\nu\big)\mathds{1}_{\rmS}(e)
+\big((\lambda_t^{k,\rmI}-\alpha)^2+c_{\rmI}^k\big)\mathds{1}_{\rmI}(e)
+(\lambda_t^{k,\rmR}-\alpha)^2\mathds{1}_{\rmR}(e).
\end{equation}
Since the control sets are bounded and all model parameters are finite, there exists a constant $C_f>0$ such that $|f^k(t, e, \alpha, \nu)| \leq C_f$ for all $t\in[0,T]$, $e\in E$, $k\in[K]$, $\alpha\in[0,1]$, and $\nu\in[0,1]$. It follows from the definition of $u_t^k(e)$ that $u_t^k(e) \in [0, TC_f]$ for all $t\in[0,T]$ and $e\in E$. Moreover, since all entries of the transition rate matrix are bounded by model formulation, there exists a constant $\bar{q} >0$, such that $\sum_{e^{\prime} \neq e} q_t^k(e, e^{\prime}, \alpha, \nu,z) \leq \bar{q}$ for all $t \in [0,T]$. Here $q_t^k(e, e^{\prime}, \ldots)$ denotes the transition rate of agent $k$ from state $e$ to $e^{\prime}$.

For a fixed $e \in E$ and $k \in [K]$, consider $t$ and $t'$ such that $0 \leq t < t' \leq T$. By the dynamic programming principle, we have
\begin{equation}
    \label{eq:dpp}
    u_t^k(e)=\inf _{(\alpha, \nu)} \mathbb{E}\Big[\int_t^{t^{\prime}} f^k\left(s, X_s^k, \alpha_s, \nu_s\right) d s+u_{t^{\prime}}^k\left(X_{t^{\prime}}^k\right) \mid X_t^k=e\Big].
\end{equation} On the other hand, consider the event $A=\left\{X_{t^{\prime}}^k \neq e| X_{t}^k = e \right\}$ that there is at least one state-jump within $[t,t']$, from results in CTMC, we have
$\mathbb{P}(A) \leq 1-e^{-\bar{q}(t' - t)} \leq \bar{q}(t' - t)$. For the first term of~\eqref{eq:dpp}, it is directly bounded by $C_f |t' - t|$; the second term satisfies the inequality $|\mathbb{E}[u_{t^{\prime}}^k(X_{t^{\prime}}^k)|X_{t}^k=e]-u_{t^{\prime}}^k(e)| \leq 2 T C_f \mathbb{P}(A) \leq 2 T C_f \bar{q} |t'-t|$. Combining both, we write
$$\left|u_t^k(e)-u_{t^{\prime}}^k(e)\right|\leq C_f |t' - t| + 2 T C_f \bar{q} |t'-t| = C_f (1 + 2T\bar{q})|t'-t|.$$

In much of the MFG literature, regularity of the value function is not emphasized: once the Hamilton–Jacobi–Bellman (HJB) equation is formulated, continuity and often $C^1$ regularity typically follows under the adopted solution concept. We nevertheless discuss this for two reasons. First, the at-most one-jump bang-bang result in Theorem~\ref{thm:charac_bang_bang} below relies on the time continuity. Second, the further introduction of jump vaccination at equilibrium leads to a state density flow whose time derivative is piecewise-continuous. In the switched-systems literature~\cite{Liberzon2003}, researchers discuss such types of solution in a Carath\'eodory sense (absolute continuity with the ODE holding almost everywhere). We therefore aim to clarify these considerations before presenting our main results. Henceforth, we assume that the social distancing guideline functions $\lambda^{k, e}:[0, T] \rightarrow (0,1]$ are continuous for all $k \in[K]$ and $e \in E$. Under this regularity, we apply classical MFG techniques to first characterize the MFNE via a forward-backward differential equations system.

\begin{theorem}
\label{thm:full_characterization}
    The multi-population mean field Nash equilibrium (MFNE) controls are written as
    \begin{equation}
    \label{eq:optcontrol}
        \begin{aligned}
\hat{\alpha}_t^k(\mathrm{S}) & =\lambda_t^{k, \mathrm{S}}+\frac{\beta^k Z_t^k\big(u_t^k(\mathrm{S})-u_t^k(\mathrm{I})\big)}{2 c_\lambda^k} \\
\hat{\alpha}_t^k(\mathrm{I}) & =\lambda_t^{k, \mathrm{I}} \\
\hat{\alpha}_t^k(\mathrm{R}) & =\lambda_t^{k, \mathrm{R}} \\
\hat{\nu}_{t}^{k} &= \mathds{1}\{\kappa^ku_{t}^k(\mathrm{S}) > c_\nu^k\}
\end{aligned}
    \end{equation}
    for all $k \in [K]$ and $t\in[0,T]$, if $(u,p)=({u}_{t}^{k},{p}_{t}^{k})_{k \in [K],t\in [0,T]}$ solves the following forward-backward ordinary differential equations (FBODE) system:
    \begin{equation}
    \label{eq:fbode}
        \begin{aligned}
\dot{p}_t^k(\mathrm{S})= & -\beta^k \hat{\alpha}_t^k(\mathrm{S}) Z_t^k p_t^k(\mathrm{S})-\kappa^k \hat{\nu}_t^k p_t^k(\mathrm{S}), \\
\dot{p}_t^k(\mathrm{I})= & \beta^k \hat{\alpha}_t^k(\mathrm{S}) Z_t^k p_t^k(\mathrm{S})-\gamma^k p_t^k(\mathrm{I}), \\
\dot{p}_t^k(\mathrm{R})= & \gamma^k p_t^k(\mathrm{I})+\kappa^k \hat{\nu}_t^k p_t^k(\mathrm{S}), \\
\dot{u}_t^k(\mathrm{S})= & \beta^k \hat{\alpha}_t^k(\mathrm{S}) Z_t^k\big(u_t^k(\mathrm{S})-u_t^k(\mathrm{I})\big)+\kappa^k \hat{\nu}_t^k\big(u_t^k(\mathrm{S})-u_t^k(\mathrm{R})\big)-c_\lambda^k\big(\lambda_t^{k, \mathrm{S}}-\hat{\alpha}_t^k(\mathrm{S})\big)^2-c_\nu^k\hat{\nu}_t^k, \\
\dot{u}_t^k(\mathrm{I})= & \gamma^k\big(u_t^k(\mathrm{I})-u_t^k(\mathrm{R})\big)-c_{\rmI}^k, \\
\dot{u}_t^k(\mathrm{R})= & 0, \\
u_T^k(e)= & 0, \quad p_0^k(e)=\pi^k_0(e), \quad \forall e \in E, \\
Z_t^k= & \sum_{l \in[K]} w(k, l) \mathbb{E}\big[\hat{\alpha}_t^l(\mathrm{I})\big]  p_t^l(\mathrm{I})m^l, \quad \forall k \in[K], \forall t \in[0, T] .
\end{aligned}
    \end{equation}
\end{theorem}

\begin{proof}
    To characterize the dynamics of $(u,p)$, we follow~\cite[Section 7.2]{carmona2018probabilistic} and write the finite-state version of the Hamilton-Jacobi-Bellman (HJB) and Kolmogorov-Fokker-Planck (KFP) equations system. We introduce the corresponding Hamiltonian for agent $k$
    \begin{equation}
    \label{eq:hmlt}
H^{k}(t, e, z, \alpha,\nu, u)=\sum_{e^{\prime} \in E} q^{k}_t(e, e^{\prime},\alpha_t,\nu_t,z_t) u_t^{k}(e^{\prime})+f^{k}(t, e, \alpha_t, \nu_t)
    \end{equation}
    By the continuity assumption on social distancing guidelines, the running cost $f^k$ is strongly convex with respect to socialization level $\alpha$ for all $k \in [K]$, and the mapping $[0,1] \ni \alpha \mapsto H^{k}(t,e,z,\alpha,\nu,u)$ admits a unique measurable minimizer $\hat{\alpha}_{t}(e),\ e \in E$. We compute it using first-order optimality condition and derive the explicit form in~\eqref{eq:optcontrol}. Since the Hamiltonian does not contain interaction term between socialization and vaccination, we similarly consider $[0,1] \ni \nu \mapsto H^{k}(t,e,z,\alpha,\nu,u)$ and apply first-order optimality condition to achieve the equilibrium vaccination rates in~\eqref{eq:optcontrol}, whose indicator form follows from the linear dependence of the running cost on vaccination. By the dynamic programming principle of optimal control, the HJB equations can be written as
    \begin{equation}
        \label{eq:hjb}
        \dot{u}^{k}_{t}(e) + H^{k}(t,e,\Delta_e u_t(\cdot),\hat{\alpha}_t(e),\hat{\nu}_t) = 0, \quad u^{k}_T(e)= 0, \quad e \in E,
    \end{equation}
    where $\Delta_e u_t(e^{\prime})=u_t(e^{\prime})-u_t(e)$. We remark that HJB is originally a partial differential equation (PDE) system; since in our model the state space is finite, it reduces to a finite-dimensional ODE system. We also argue that the population state distribution flow $({p}_t)_{t\in[0,T]}$ solves the KFP equation
    \begin{equation}
        \label{eq:kfp}
        \dot{p}^{k}_t(e)=\sum_{e^{\prime} \in E} q_t^{k}(e^{\prime}, e, \hat{\alpha}^k_{t}( e^{\prime}),\hat{\nu}^k_t;Z_t^k) p^{k}_t(e^{\prime}), \quad p^{k}_0(e) = \pi^{k}_{0}(e), \quad e \in E.
    \end{equation}
    At equilibrium, the KFP equations are coupled with the HJB equations via controls: the equilibrium controls depend on the value function $u$ and on the mean-field interaction term $Z$, and $Z$ itself depends on the state density $p$. We thereby obtain the FBODE system stated in~\eqref{eq:fbode}.
\end{proof}

The above result states that once the FBODE system is solved, the equilibrium vaccination rates, socialization levels, and the disease dynamics at the MFNE can be determined. To further explore the MFNE collective behavior, we state some regularity conditions to ensure a non-trivial epidemics problem.
\begin{assumption}[Model Regularity]
    \label{assu:model_regularity}
    Out of model formulation, we assume that for each group $k \in [K]$, there exists some group $l(k) \in [K]$ such that
    $$\pi_0^{\ell(k)}(\mathrm{I})>0, \quad w(k,l(k)) > 0.$$
\end{assumption}
We remark that a stronger but simpler assumption is to pose positive initial infection proportion and strictly positive within-group connection strength for all groups, i.e. $\pi_0^k(\rmI) > 0$ and $w(k,k) > 0$ for any $k \in [K]$. Under this regularity assumption, we claim that the aggregate variable remains strictly positive on $[0,T]$.

\begin{lemma}
    \label{lemma:pos_Z}
    Under Assumption~\ref{assu:model_regularity}, along any solution of the FBODE, the aggregate variable is strictly positive, i.e. $Z_t^k > 0$ for all $k \in [K]$ and $t \in [0,T]$.
\end{lemma}
\begin{proof}
    Fix $k \in [K]$ and consider $l(k)$ satisfying Assumption~\ref{assu:model_regularity}. By variation of constants,
$$
p_t^{l(k)}(\mathrm I)
=
e^{-\gamma^{l(k)} t}\pi_0^{l(k)}(\mathrm I)
+
\int_0^t e^{-\gamma^{l(k)} (t-s)}
\beta^{l(k)} \alpha_s^{l(k)}(\rmI) p_s^{l(k)}(\mathrm S) Z_s^{l(k)}\,ds .
$$
Since all terms on the right-hand side are nonnegative and
$e^{-\gamma^{l(k)} t}\pi_0^{l(k)}(\mathrm I)>0$, we obtain $p_t^{l(k)}(\mathrm I)>0, \ \forall t\in[0,T]$.
Using the representation of $Z_t^k$ and the nonnegativity of the coefficients, it follows
$
Z_t^k=\sum_{l\in [K]} w(k,l) \lambda_t^{l,\rmI} p_t^{l}(\mathrm I) m^l
\geq
w(k,l(k)) \lambda_t^{l(k),\rmI} p_t^{l(k)}(\mathrm I) m^{l(k)}>0$.
\end{proof}

\begin{theorem}
    \label{thm:charac_bang_bang}
    Under Assumption~\ref{assu:model_regularity}, the equilibrium vaccination rates at the MFNE follow a bang-bang strategy with at most one jump: $$\hat{\nu}^k_t= \begin{cases}1 & \text { if } t<t_1^k, \\ 0 & \text { if } t \geq t_1^k,\end{cases}$$ 
    where the jump time $t_1^k$ satisfies $u_{t_1^k}^k(\mathrm{S}) = \frac{c_{\nu}^{k}}{\kappa^{k}}$. In particular, if $u_{0}^k(\mathrm{S}) < \frac{c_{\nu}^{k}}{\kappa^{k}}$, then $t_1^k = 0$, i.e. there is no jump and $\hat{\nu}_t^k = 0$ for all $t \in [0,T]$.
\end{theorem}

\begin{proof}
    For the sake of simplification in presentation, we present the proof in one-population setting, i.e., $K=1$. The technical results extend to the general case in straightforward manner. Under our model setting, we first note $u_t(\rmR) = 0$ for any $t \in [0,T]$, since recovered individuals follow their guidelines. Minimizing the HJB equation gives the dynamics of value functions for infectious $\dot{u}_t(\mathrm{I})=\gamma u_t(\mathrm{I})-c_{\rmI}$ with $u_T(\rmI)=0$.  This has the following explicit solution \begin{equation}
        \label{eq:solnI}
        u_{t}(\rmI) = \frac{c_{\rmI}}{\gamma}(1-e^{-\gamma (T-t)}), \quad t \in [0,T].
    \end{equation} 
    We notice $u_t(\rmI)$ is monotonically decreasing on $[0,T]$ with $u_T(\rmI) = 0$ at the terminal time. On the other hand, if we write back the dynamics of $u_t(\rmS)$ as in the Hamiltonian minimization form of HJB, it follows
    \begin{equation}
        \label{eq:hjbus}
        -\dot{u}_t(\rmS) = \inf_{\alpha \in [0,1], \nu \in [0,1]} \left\{ c_{\lambda}(\lambda^{\rmS}_{t} - \alpha)^2 + c_{\nu} \nu + \beta \alpha Z_{t}(u_t(\rmI) - u_t(\rmS)) + 
    \kappa \nu (u_t(\rmR) - u_t(\rmS))\right\}
    \end{equation}
     By the property of infimum, we plug in admissible controls (not necessarily minimizers) such that $\alpha = \lambda_t$ and $\nu = 0$ for arbitrary $t \in [0,T]$ and achieve a relaxed bound \begin{equation}
        \label{eq:boundus}
        -\dot{u}_{t}(\rmS) \leq \beta \lambda^{\rmS}_t Z_t (u_t(\rmI) - u_t(\rmS))
    \end{equation}
    Naturally $-\dot{u}_{T}(\rmS) \leq 0$ since $u_T(\rmI) = u_T(\rmS) = 0$. Also from~\eqref{eq:solnI}, $-\dot{u}_T(\rmI) = c_{\rmI} > 0$. Therefore, it follows \begin{equation}
        \label{eq:dom_near_T}
        u_{T-}(\rmS) < u_{T-}(\rmI)
    \end{equation} By continuity of value functions, we will prove $u_t(\rmS) < u_t(\rmI)$ for all $t \in [0,T)$ by using proof  by contradiction.
    Assume $u_t(\rmS) < u_t(\rmI)$ does not hold for some $t \in [0,T)$ and let $\tau \in [0,T)$ be the largest time point at which $u_t(\rmS) = u_t(\rmI)$. Recall that $u_t(\rmI)$ is decreasing on $[\tau, T)$. So by \eqref{eq:dom_near_T} and the continuity of value functions, considering the backward direction from $T$, for $u_t(\rmS)$ to hit $u_{t}(\rmI)$ at time $\tau$, there must have $\dot{u}_{\tau}(\rmS) \leq \dot{u}_{\tau}(\rmI) < 0$. This contradicts with $-\dot{u}_{\tau}(\rmS) \leq 0$ from \eqref{eq:hjbus} since at $\tau$, $u_\tau(\rmS)=u_\tau(\rmI)$.

    Since there is no interaction term of the two controls (socialization level and vaccination rate of susceptible individuals) in the Hamiltonian~\eqref{eq:hjbus}, it is immediate that $\hat{\nu}_{t} = 1$ if $c_{\nu} < \kappa u_t(\rmS)$ and $\hat{\nu}_{t} = 0$ if $c_{\nu} > \kappa u_t(\rmS)$. It is left to show that there is at most one jump, i.e., $u_t(\mathrm{S})$ cross $c_v/\kappa$ at most once on $[0,T]$. Let $t_1$ be the first time $u_t(\mathrm{S}) = c_v/\kappa$, at which we choose $\hat{\nu}_{t_1} = 0$ by convention. Then by \eqref{eq:hjbus}, 
    \begin{equation}
        \label{eq:neg_slope_crossing}
        -\dot{u}_{t_1}(\rmS) = \inf_{\alpha \in [0,1]} \left\{ c_{\lambda}(\lambda^{\rmS}_{t_1} - \alpha_{t_1})^2 + \beta \alpha_{t_1} Z_{t_1}(u_{t_1}(\rmI) - u_{t_1}(\rmS)) + 0 \right\} > 0
    \end{equation}
    The equality holds by definition of $t_1$, and the (strictly positive) inequality holds since $u_{t_1}(\rmS) < u_{t_1}(\rmI)$ and also by applying the Lemma~\ref{lemma:pos_Z}.
    Therefore, by continuity of value functions, we conclude that $u_t(\rmS)$ is strictly dominated by $c_{\nu}/\kappa$ on $[t_1,T]$. For the sake of contradiction, suppose that there exists another crossing time $t_2 \in (t_1,T)$ such that $u_{t_2}(\rmS)=c_{\nu}/\kappa$. By~\eqref{eq:neg_slope_crossing}, we similarly obtain $\dot{u}_{t_2}(\rmS)<0$. On the other hand, since $u_t(\rmS)$ lies strictly below $c_{\nu}/\kappa$ immediately after $t_1$ due to $\dot{u}_{t_1}(\rmS)<0$, the function $u_t(\rmS)$ must reach the threshold $c_{\nu}/\kappa$ from below at time $t_2$, which requires $\dot{u}_{t_2}(\rmS)\geq 0$. This is a contradiction and thereby completes the proof. 
\end{proof}

The above result states that at MFNE, individuals vaccinate until a specific time determined by the cost and the efficacy of the vaccination and they stop vaccinating. Furthermore, it says that if the cost of vaccination is relatively higher then the vaccination efficacy, there is no vaccination.

The next result establishes the existence of the solution the FBODE system. Because the equilibrium vaccination rate is of indicator jump form, the commonly used contraction arguments based on Banach’s fixed point theorem are not directly applicable, unlike in settings with continuous equilibrium controls (see e.g.,~\cite{liu2025incorporatingauthorityperceptioneconomic, multi_major_minor}). Our approach is therefore to fix the jump times a priori and to study the resulting reduced system; we then recover the jump times via a consistency condition to obtain a solution of the original FBODE. For clarity, we refer to the reduced system with the exogenously fixed jump times as the \emph{sub-FBODE} and to the original system as the \emph{full-FBODE}. Before presenting the proof of existence, we state and discuss several technical lemmas that underpin this construction.

\begin{lemma}[Well-posedness of sub-FBODE]
\label{lemma:sub_wellposedness}
Fix a vaccination jump-time vector $t_1 = (t_1^k)_{k \in [K]} \in[0,T]^K$ and let $\tilde\nu_t^k=\mathds{1}_{\{t<t_1^k\}}$. For sufficiently small $T>0$, the corresponding sub-FBODE system admits a unique Carath\'eodory solution $(u^{t_1},p^{t_1})$ on $[0,T]$.
\end{lemma}
The proof of Lemma~\ref{lemma:sub_wellposedness} can be found in Appendix~\ref{sec:appendix_sub_wellposedness}.
Lemma~\ref{lemma:sub_wellposedness} states that for each prescribed jump-time vector $t_1\in[0,T]^K$, the corresponding sub-FBODE system is well-posed and admits a unique bounded solution $(u^{t_1},p^{t_1})$. We next establish the continuous dependence of the induced susceptible value function, $u^{t_1}(\rmS)$, on $t_1$ (in a sense of uniform norm), a property that will be used to verify continuity of the jump-time update map in the subsequent fixed-point argument.

\begin{lemma}[Continuity of $u^{t_1}$ with respect to $t_1$]
\label{lemma:sub_continuity}
Let $t_1\in[0,T]^K$ be fixed and let $(u^{t_1},p^{t_1})$ be the unique solution of the corresponding sub-FBODE. Then the map
$$
[0,T]^K\ni t_1 \longmapsto u^{t_1}(\rmS)\in C([0,T];\mathbb{R}^K)
$$
is continuous with respect to the uniform norm.
\end{lemma}
The proof of Lemma~\ref{lemma:sub_continuity} can be found in Appendix~\ref{sec:appendix_sub_continuity}. With the continuous dependence result established, we can proceed to reconstruct the equilibrium vaccination control via the threshold rule. A remaining technical point is that, for the reduced system with prescribed jump times, the susceptible value function does not necessarily satisfy the same structural properties as in the full-FBODE; in particular, the one-jump characterization proved in Theorem~\ref{thm:charac_bang_bang} does not follow for free. Since at-most-once threshold-crossing time is needed to define the jump-time update map, we next provide a condition under which the sub-FBODE value function still generates a vaccination control with at most one jump.

\begin{lemma}[At-most one-crossing property for the sub-FBODE]
\label{lemma:sub_one_jump}
Fix $t_1\in[0,T]^K$ and let $u^{t_1}$ be the corresponding sub-FBODE solution. If $c_\nu^k<c_{\rmI}^k$ for all $k\in[K]$, then for each $k$ the function $t\mapsto u_t^{t_1,k}(\rmS)$ meets the threshold $c_\nu^k/\kappa^k$ at most once on $[0,T]$ under short time condition.
\end{lemma}

The proof of Lemma~\ref{lemma:sub_one_jump} can be found in Appendix~\ref{sec:appendix_sub_one_jump}. The above lemma gives an epidemiologically natural condition for the at-most one threshold-crossing. It states that when the cost of vaccination is lower than the cost of infection, once the susceptible value function drops below the vaccination threshold, it does not rise above it again. 

Next, we construct the hitting-time map, which associates to the (unique) time at which the susceptible value function path reaches the vaccination threshold. Under the at-most one-crossing condition for the susceptible individual's value function, this time is well-defined. Then, we show that the resulting map is continuous in the value function under the supremum norm, ensuring that small perturbations of the value function lead to small perturbations of the induced switching time.

\begin{lemma}[Continuity of the hitting-time map]
\label{lemma:jump_time_continuity}
Let $c=(c^k)_{k\in[K]}\in[0,\infty)^K$. Consider the space
\begin{equation*}
    \begin{aligned}
        \mathcal{S}
:=
\Bigl\{
A\in & AC([0,T];\mathbb{R}^K):\;
A_T^k=0,\
\\&\#\{t\in[0,T]:A_t^k=c^k\}\leq 1, \ \forall k\in[K], 
\text{ and whenever }A_{t'}^k=c^k,\dot A_{t'}^k<0
\Bigr\}.
    \end{aligned}
\end{equation*}
For $A\in\mathcal{S}$, define $\tilde t(A)\in[0,T]^K$ by $t(A):=(\tilde t_1(A),\ldots,\tilde t_K(A))^\top$, where
$$
\tilde t_k(A):=
\begin{cases}
t_k, & \text{if there exists a unique }t_k\in[0,T]\text{ such that }A_{t_k}^k=c^k,\\
0, & \text{otherwise}.
\end{cases}
$$
Then $\tilde t:\mathcal{S}\to[0,T]^K$ is continuous when $\mathcal{S}$ is equipped with $\|A\|_T:=\sup_{t\in[0,T]}\|A_t\|_2$ and $[0,T]^K$ is equipped with $\|x\|_\infty:=\max_k|x_k|$.
\end{lemma}
The proof of Lemma~\ref{lemma:jump_time_continuity} can be found in Appendix~\ref{sec:appendix_jump_time_continuity}. We are ready now state a well-posedness result for the FBODE system~\eqref{eq:fbode}.

\begin{theorem}[Existence for the FBODE]
\label{thm:fbode_wellposedness}
Assume the \emph{short-time condition} holds, namely $T>0$ is sufficiently small so that the results in Lemmas~\ref{lemma:sub_wellposedness},~\ref{lemma:sub_continuity} and~\ref{lemma:sub_one_jump} apply. If $c_{\nu}^{k} < c_{\rmI}^{k}$ for all $k\in[K]$, then the FBODE system~\eqref{eq:fbode} admits at least one Carath\'eodory solution $(\boldsymbol{u},\boldsymbol{p})$ on $[0,T]$.
\end{theorem}

\begin{proof}
Fix an arbitrary vaccination jump-time vector $t_1=(t_1^k)_{k\in[K]}\in[0,T]^K$ and consider the induced vaccination
$$
\tilde\nu_t^{t_1,k}:=\mathds{1}_{\{t<t_1^k\}}, \qquad t\in[0,T],\ k\in[K].
$$
We adopt the right-continuous convention $\tilde\nu_{t_1^k}^{t_1,k}=0$. Since the value of the control at a single time point does not affect the integral formulation of the ODE system, this convention is without loss of generality and is used only to simplify the notation. Let $(u^{t_1},p^{t_1})$ denote the unique solution of the corresponding sub-FBODE system, i.e., system \eqref{eq:fbode} where $\hat\nu$ replaced by $\tilde\nu^{t_1}$. Existence and uniqueness for small $T$ are given by Lemma~\ref{lemma:sub_wellposedness}. From $u^{t_1}(\rmS):=(u_t^{t_1,k}(\rmS))_{t\in[0,T],\,k\in[K]}$, define a new jump-time vector $\tilde t_1=f(t_1)\in[0,T]^K$ by
$$
\tilde t_{1,k}:=
\begin{cases}
t_k, & \text{if there exists a unique }t_k\in[0,T]\text{ such that }u_{t_k}^{t_1,k}(\rmS)=c_\nu^k/\kappa^k,\\[0.7em]
0, & \text{otherwise}.
\end{cases}
$$
By Lemma~\ref{lemma:sub_one_jump}, each component $t\mapsto u_t^{t_1,k}(\rmS)$ hits the threshold $c_\nu^k/\kappa^k$ at most once, so $\tilde t_1$ is well-defined. We now check that $f$ is continuous. Lemma~\ref{lemma:sub_continuity} yields continuity of the map $t_1\mapsto u^{t_1}(\rmS)$ in the uniform norm.

Then, to verify the continuity of the map $u^{t_1}(\rmS)\mapsto \tilde t_1$, we need to apply Lemma~\ref{lemma:jump_time_continuity}. To apply Lemma~\ref{lemma:jump_time_continuity}, it remains to check if $u_\tau^{t_1,k}(\rmS)=c_\nu^k/\kappa^k$ for some $k\in[K]$ and $\tau\in[0,T]$,
then $\dot u_\tau^{t_1,k}(\rmS)<0$. At time $\tau$, the HJB equation gives
$$
-\dot{u}_\tau^{t_1,k}(\rmS)
=
\inf_{\alpha\in[0,1]}
\left\{
c_\lambda^k\bigl(\lambda_\tau^{k,\rmS}-\alpha\bigr)^2
+
\beta^k \alpha Z_\tau^k \bigl(u_\tau^{t_1,k}(\rmI)-u_\tau^{t_1,k}(\rmS)\bigr)
\right\}.$$
By the positivity assumption on the guidelines and the analogue of Lemma~\ref{lemma:pos_Z} for the sub-FBODE, we have $\lambda_\tau^{k,\rmS}>0$ and $Z_\tau^k>0$. Furthermore, $\beta^k>0$, $c_\lambda^k>0$, and Lemma~\ref{lemma:sub_one_jump} implies that $u_\tau^{t_1,k}(\rmI)>u_\tau^{t_1,k}(\rmS)$. Hence, we have $\dot u_\tau^{t_1,k}(\rmS)<0$.

Therefore, the fixed point mapping $f:[0,T]^K \to [0,T]^K$ is continuous; since $[0,T]^K$ is non-empty, convex and compact, Brouwer fixed point theorem guarantees the existence of $\tau^* \in [0,T]^K$ such that $\tau^*=f(\tau^*)$. The corresponding unique sub-FBODE solution $(u^{\tau^*}, p^{\tau^*})$ is the full-FBODE solution. This completes the proof.
\end{proof}

Intuitively, Theorem~\ref{thm:fbode_wellposedness} states that MFNE socialization levels and vaccination rates exist if the cost of vaccination is lower than the cost of infection for each group $k\in[K]$ and the time horizon is short enough.

\section{Model with Population-awareness}
\label{sec:pop-aware}
During a large-scale pandemic (such as the recent COVID-19 pandemic), people may acquire infection information of the world or their country via internet, and this might affect their behavior by incentivizing them to take precautions in advance. Therefore, we consider an extension of the baseline MFG model in Section~\ref{sec:mfg_res} by incorporating the population infection information into the agents' cost objective. Intuitively, it represents extra awareness cost that an individual is perceiving from the population-level disease spread information. Since our model does not include waning of immunity rates, individuals in state R stay immune for the disease and therefore, perceive no awareness cost from the population. Consider the composite infected proportion $P_t(\mathrm{I})=\sum_{l=1}^K m^l p_t^l(\mathrm{I})$. The cost objective of representative agent $k$, for $k\in[K]$ under the new modeling regime is defined as
\begin{equation}
\begin{aligned}
J^k(\boldsymbol{\alpha}^k, \boldsymbol{\nu}^k ; \boldsymbol{Z})&=\mathbb{E}\bigg[\int_{0}^{T} \Big\{\big(c_\lambda^k(\lambda_t^{k, \mathrm{S}}-\alpha_t^k)^2+c_\nu^k\nu_t^k+c_p^{k,\rmS} g(P_t(\rmI))\big) \mathds{1}_{\mathrm{S}}(X_t^k) \\
& +\big((\lambda_t^{k, \mathrm{I}}-\alpha_t^k)^2+c_{\mathrm{I}}^k+c_p^{k,\rmI} g(P_t(\rmI))\big) \mathds{1}_{\mathrm{I}}(X_t^k)+(\lambda_t^{k, \mathrm{R}}-\alpha_t^k)^2 \mathds{1}_{\mathrm{R}}(X_t^k) d t\bigg]
\end{aligned}
\end{equation}
Here, $c_p^{k,\rmS} > 0$ and $c_p^{k,\rmI} > 0$ are the scaling coefficients for the population-awareness cost terms for representative agent $k$ when she is in state S and state I, respectively. The function $g:[0,1] \rightarrow \mathbb{R}_{+}$ is a non-decreasing awareness cost function, capturing the idea that higher infection prevalence induces greater concern and thus higher awareness cost. In this work, we take $g(x)=x$ for simplicity, which keeps term on a comparable scale with the other components of the objective.

\begin{theorem}
\label{thm:awareness_full_charac}
    The multi-population mean field Nash equilibrium controls with population-awareness are written as
    \begin{equation}
    \label{eq:aware_optcontrol}
        \begin{aligned}
\hat{\alpha}_t^k(\mathrm{S}) & =\lambda_t^{k, \mathrm{S}}+\frac{\beta^k Z_t^k\big(u_t^k(\mathrm{S})-u_t^k(\mathrm{I})\big)}{2 c_\lambda^k} \\
\hat{\alpha}_t^k(\mathrm{I}) & =\lambda_t^{k, \mathrm{I}} \\
\hat{\alpha}_t^k(\mathrm{R}) & =\lambda_t^{k, \mathrm{R}} \\
\hat{\nu}_{t}^{k} & = \mathds{1}\{{\kappa^k u_{t}^k(\mathrm{S}) > c_{\nu}^k}\}
\end{aligned}
    \end{equation}
    for all $k \in [K]$, $t\in[0,T]$, if $(u,p)=(u_{t}^{k}, p_{t}^{k})_{k \in [K],t\in [0,T]}$ solves the following FBODE system:
    \begin{equation}
    \label{eq:fbode_pop_aware}
        \begin{aligned}
\dot{p}_t^k(\mathrm{S})= & -\beta^k \hat{\alpha}_t^k(\mathrm{S}) Z_t^k p_t^k(\mathrm{S})-\kappa^k \hat{\nu}_t^k p_t^k(\mathrm{S}), \\
\dot{p}_t^k(\mathrm{I})= & \beta^k \hat{\alpha}_t^k(\mathrm{S}) Z_t^k p_t^k(\mathrm{S})-\gamma^k p_t^k(\mathrm{I}), \\
\dot{p}_t^k(\mathrm{R})= & \gamma^k p_t^k(\mathrm{I})+\kappa^k \hat{\nu}_t^k p_t^k(\mathrm{S}), \\
\dot{u}_t^k(\mathrm{S})= & \beta^k \hat{\alpha}_t^k(\mathrm{S}) Z_t^k\big(u_t^k(\mathrm{S})-u_t^k(\mathrm{I})\big)+\kappa^k \hat{\nu}_t^k\big(u_t^k(\mathrm{S})-u_t^k(\mathrm{R})\big) \\ & \hspace{4cm}-c_\lambda^k\big(\lambda_t^{k, \mathrm{S}}-\hat{\alpha}_t^k(\mathrm{S})\big)^2-c_\nu^k\hat{\nu}_t^k - c_{p}^{k,\rmS} P_t(\mathrm{I}), \\
\dot{u}_t^k(\mathrm{I})= & \gamma^k\big(u_t^k(\mathrm{I})-u_t^k(\mathrm{R})\big)-c_{\rmI}^k - c_{p}^{k,\rmI} P_t(\mathrm{I}), \\
\dot{u}_t^k(\mathrm{R})= & 0, \\
u_T^k(e)= & 0, \quad p_0^k(e)=\pi^k_0(e), \quad \forall x \in E, \\
Z_t^k= & \sum_{l \in[K]} w(k, l) \mathbb{E}[\hat{\alpha}_t^l(\mathrm{I})] p_t^l(\rmI) m^l, \quad \forall k \in[K], \forall t \in[0, T] .
\end{aligned}
    \end{equation}
\end{theorem}
\begin{proof}
    The proof proceeds similarly as Theorem~\ref{thm:full_characterization} and is therefore omitted for brevity.
\end{proof}

We then derive a sufficient parameter condition under which the equilibrium vaccination rate retains its at-most one-jump structure in the population-awareness setting. In addition, we compare the baseline model and its population-awareness extension in terms of their best response structures.
\begin{theorem}
\label{thm:aware_bang_bang}
    If $c_{\rmI}^k > c_p^{k,\rmS}$ for all $k \in [K]$, then the equilibrium vaccination rates also follow a bang-bang strategy with at most one jump under short-time condition. In addition, if the mean-field interaction, $Z$, is fixed, the best response vaccination jump time in population-awareness extension is no less than that in baseline model, with all other parameters stay the same.
\end{theorem}

Theorem~\ref{thm:aware_bang_bang} states that if the population-awareness cost coefficient for susceptible individuals is smaller than their infection costs, then they vaccinate until a specific time and then they stop vaccinating. Furthermore, if people are aware of the population-level infection via their cost functional, then they decide to vaccinate for longer durations under same population conditions.

\section{Numerical Study}
\label{sec:numerics}
In this section, we first propose an algorithm for numerically computing the multi-population MFNE, and then implement both the baseline model and its population-aware extension. We first conduct a single-population simulation study to illustrate the equilibrium vaccination behavior and the impact of incorporating population-awareness to visualize our theoretical findings, and then conduct a multi-population real-data informed application. Throughout the implementation, we keep the equilibrium vaccination rate as an indicator function based on the value functions without assuming an at-most one-jump structure. The resulting numerical solutions confirm the at-most one-jump structures established in Theorem~\ref{thm:charac_bang_bang} and Theorem~\ref{thm:aware_bang_bang}. Finally, we calibrate the model using socio-economic parameters informed by real data from~\cite{liu2025incorporatingauthorityperceptioneconomic}, and conduct a data-driven study to quantify policy-relevant effects.

Following the theoretical results, in order to find a multi-population MFNE for our model, we will need to solve the FBODE system given in Theorem~\ref{thm:full_characterization} for the base model and similarly in Theorem~\ref{thm:awareness_full_charac} for the population-awareness extension. We discretize the continuous time FBODE system by Euler scheme in time and solve the resulting discrete system by fixed-point iterations. The state distribution is solved forward in time from the given initial condition, while the value function is solved backward in time from the terminal condition, which is $0$ for each state. Moreover, because the aggregate term depends only on the state distribution and the model parameters, it can be computed separately from the equilibrium socialization and vaccination controls. We therefore update these variables sequentially and repeat the procedure until convergence. The details of the numerical approach are given in Algorithm~\ref{algo:vacc}.

\begin{algorithm}[ht]
\caption{Multi-population MFNE \label{algo:vacc}}
\textbf{Input:}
\begin{itemize}[leftmargin=*, nosep]
    \item Time horizon $T$; Time increments $\Delta t$; Initial state distributions $\pi_0^k(\cdot)$ for $k\in[K]$
    \item Model parameters for $k \in [K]$: 
          $\beta^k, \kappa^k, \gamma^k, c_{\lambda}^k, c_{\nu}^k, c_{\rmI}^k$. 
    \item Social-distancing guidelines $\lambda^k = (\lambda_t^{k,e})_{t\in\{0, \Delta t, 2\Delta t, \dots, T\}=:[T],e\in E}$
    \item Group connection matrix $(w(k,l))_{k,l \in [K]}$; Convergence tolerance $\epsilon$
\end{itemize}
\vskip1mm

\textbf{Output:} MFNE socialization levels and vaccination rates for representative agents in each group $k\in[K]$: $\hat{\boldsymbol{\alpha}}^k$, $\hat{\boldsymbol{\nu}}^k$; MFNE state densities in each group $k\in[K]$: $\bp^k$.

\vskip2mm

\begin{algorithmic}[1]
\STATE Initialize ${\bp}^{k,(0)}(e)=({p}_t^{k,(0)}(e))_{t \in [T]}$ and ${\bu}^{k,(0)}(e)=({u}_0^{k,(0)}(e))_{t\in [T]}$ \\ for $k\in[K], e\in E$. Set the iteration counter at $j = 1$. \vskip2mm
\WHILE{$\lVert {\bp}^{k,(j)}- {\bp}^{k,(j-1)}\rVert>\epsilon$ \OR $\lVert {\bu}^{k,(j)}- {\bu}^{k,(j-1)}\rVert>\epsilon$, 
 for any $k\in[K]$}\vskip2mm
    \STATE{Calculate $\boldsymbol{Z}^{k,(j+1)} = (Z_t^{k, (j+1)})_{t\in[T]}$ using the last equation in FBODE~\eqref{eq:fbode} and the second equation in~\eqref{eq:optcontrol}.}\vskip1.5mm
    \STATE{Calculate best response socialization levels $\boldsymbol{\alpha}^{k,(j)} = (\alpha_t^{k,(j)}(e))_{t\in[T],e\in E}$ with the first three equations in~\eqref{eq:optcontrol}, and best response vaccination rate $\boldsymbol{\nu}^{k,(j+1)} = (\nu_t^{k,(j+1)})_{t \in [T]}$ with $\nu_t^{k,(j)} = \mathds{1}\{\kappa^k u_{t}^{k}(\rmS) > c_{\nu}^k\}$ by using $\boldsymbol{u}^{k,(j)}$, $\boldsymbol{Z}^{k,(j+1)}$.}\vskip1.5mm
    \STATE{Update $\bp^{k, (j+1)}$ by solving the forward ODEs in~\eqref{eq:fbode} by using $\boldsymbol{\alpha}^{k,(j+1)}, \boldsymbol{\nu}^{k,(j+1)}, \bZ^{k,(j+1)}, \bu^{k,(j)}$.}\vskip1.5mm
    \STATE{Update $\bu^{k, (j+1)}$ by solving the backward ODEs in~\eqref{eq:fbode} by using $\boldsymbol{\alpha}^{k,(j+1)}, \boldsymbol{\nu}^{k,(j+1)}, \bZ^{k,(j+1)}, \bp^{k,(j+1)}$.}\vskip1.5mm
\ENDWHILE
\vskip2mm
\STATE{Calculate $\hat{\bZ}^{k}$ for all $k\in[K]$ with the last equation in FBODE~\eqref{eq:fbode} and the second equation in~\eqref{eq:optcontrol}.}\vskip2mm
\STATE{Calculate best response socialization levels $(\hat{\alpha}_t^{k}(e))_{t\in[T],e\in E}$ by plugging in $\bu^{k,(j)}$, $\hat{\bZ}^{k}$ into the first three equations in~\eqref{eq:optcontrol}, and best response vaccination rate with $\hat{\nu}_t^{k} = \mathds{1}\{\kappa^k u_{t}^{k,(j)}(\rmS) > c_{\nu}^k\}$ for all $k\in[K]$.} \vskip2mm
\RETURN $(\hat{\boldsymbol{\alpha}}^k, \hat{\boldsymbol{\nu}}^k, \bp^{k,(j)})_{k\in[K]}$.
\end{algorithmic}
\end{algorithm}

\subsection{Single-population Simulation Study}
Our first numerical experiment is a sensitivity analysis study aimed at illustrating the structure of the model and the effect of population-awareness parameters on the MFNE. To this end, the multi-population network is deactivated, and we focus on a single-population setting for simplification in presentation. Parameters that are not the primary objects of interest are fixed at empirically motivated values, and are summarized in Table~\ref{tab:param_single_pop_compact}. Vaccination parameters $c_{\nu}$ and $\kappa$ are rescaled from the ones in~\cite{liu2025incorporatingauthorityperceptioneconomic} to fit the present linear vaccination cost formulation. 
We choose $\beta$ and $\gamma$ based on initial SARS-CoV-2 estimates. We take $\beta = 0.4$ as the average transmission level~\cite{Liu2020}, and $\gamma=1/7$ represents the recovery rate of 7 days on average~\cite{VanKampen2021}. We set the horizon $T=80$, corresponding to an 80-day operating period after the infection is first identified.

\begin{table}[htbp]
\centering
\renewcommand{\arraystretch}{1.25}
\begin{tabular}{|l l l|}
\hline
$T = 80$ 
& $\Delta t = 0.016$ 
& $\epsilon = 0.1$ \\

$\beta = 0.4$ 
& $\gamma = 1/7$ 
& $\kappa = 0.005$ \\

$c_\lambda = 1.0$ 
& $c_{\mathrm{I}} = 1.0$ 
& $c_\nu = 0.001$ \\

$\lambda^\rmS = 0.9$ 
& $\lambda^\rmI = 0.9$ 
& $\lambda^\rmR = 0.9$ \\

$\pi_0(\rmS) = 0.99$
& $\pi_0(\rmI) = 0.01$
& $\pi_0(\rmR) = 0$ \\

\hline
\end{tabular}
\caption{Single-population numerical experiment: parameter configuration.}
\label{tab:param_single_pop_compact}
\end{table}

We compare the MFNE behaviors and state distributions across three different values of the population-awareness coefficient, namely $0$, $0.1$, and $0.5$. Particularly, the case where the population-awareness coefficient is equal to 0 corresponds to the baseline MFG model introduced in Section~\ref{sec:mfg_res}. Figure~\ref{fig:p1} presents the equilibrium vaccination rates under the MFNE. The left panel validates the at-most one-jump bang-bang structure of the equilibrium vaccination rate, while the right panel, together with Figure~\ref{fig:p2}, shows that a larger population-awareness coefficient leads individuals to sustain vaccination for a longer period. Figure~\ref{fig:p3} further reveals an important \textit{trade-off} between equilibrium vaccination and socialization decisions. In particular, simulations with a higher population-awareness coefficient generate higher equilibrium socialization levels. This partially offsets the longer vaccination duration induced by higher population-awareness and results in a slightly higher infection proportion.

\begin{figure}[ht]
    \centering
    \includegraphics[width=0.9\linewidth]{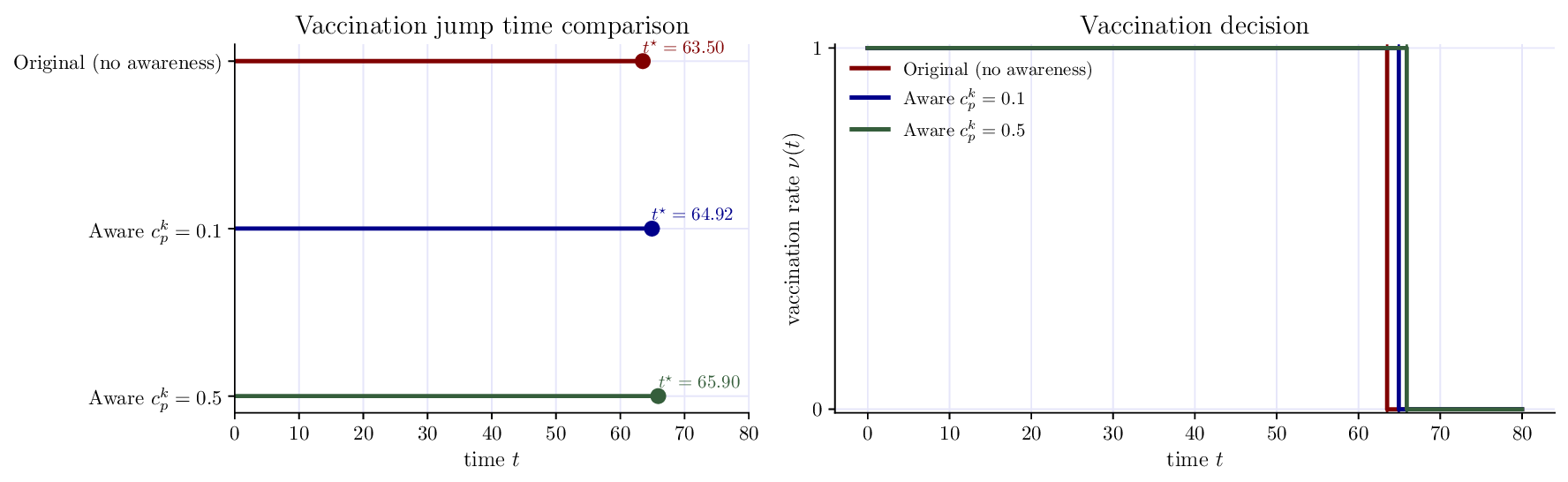}
    \caption{\textbf{Single Population Simulation:} comparison of equilibrium vaccination rate jump times (left) and vaccination rates (right, illustrating the one-jump structure) across models with different population-aware coefficient levels.}
    \label{fig:p1}
\end{figure}

\begin{figure}[ht]
    \centering
    \includegraphics[width=0.9\linewidth]{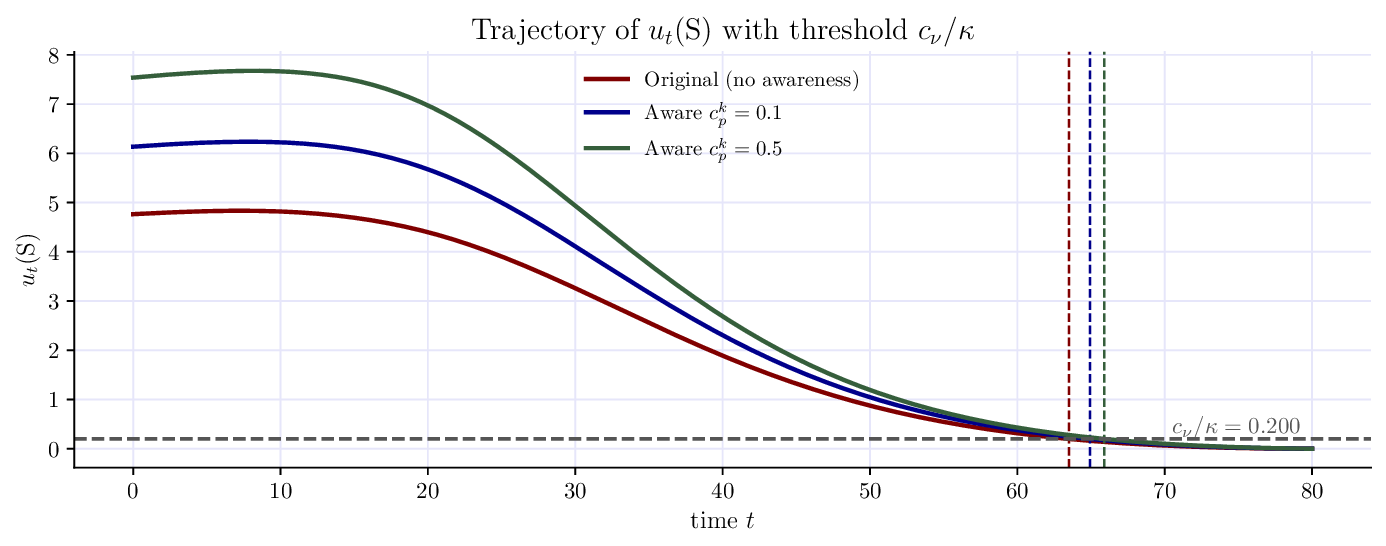}
    \caption{\textbf{Single Population Simulation:} comparison of susceptible state value function trajectories on $[0,T]$ across models with different population-aware coefficient levels. Dashed line: the vaccination jump threshold $c_\nu/\kappa$.}
    \label{fig:p2}
\end{figure}

\begin{figure}[ht]
    \centering
    \includegraphics[width=0.9\linewidth]{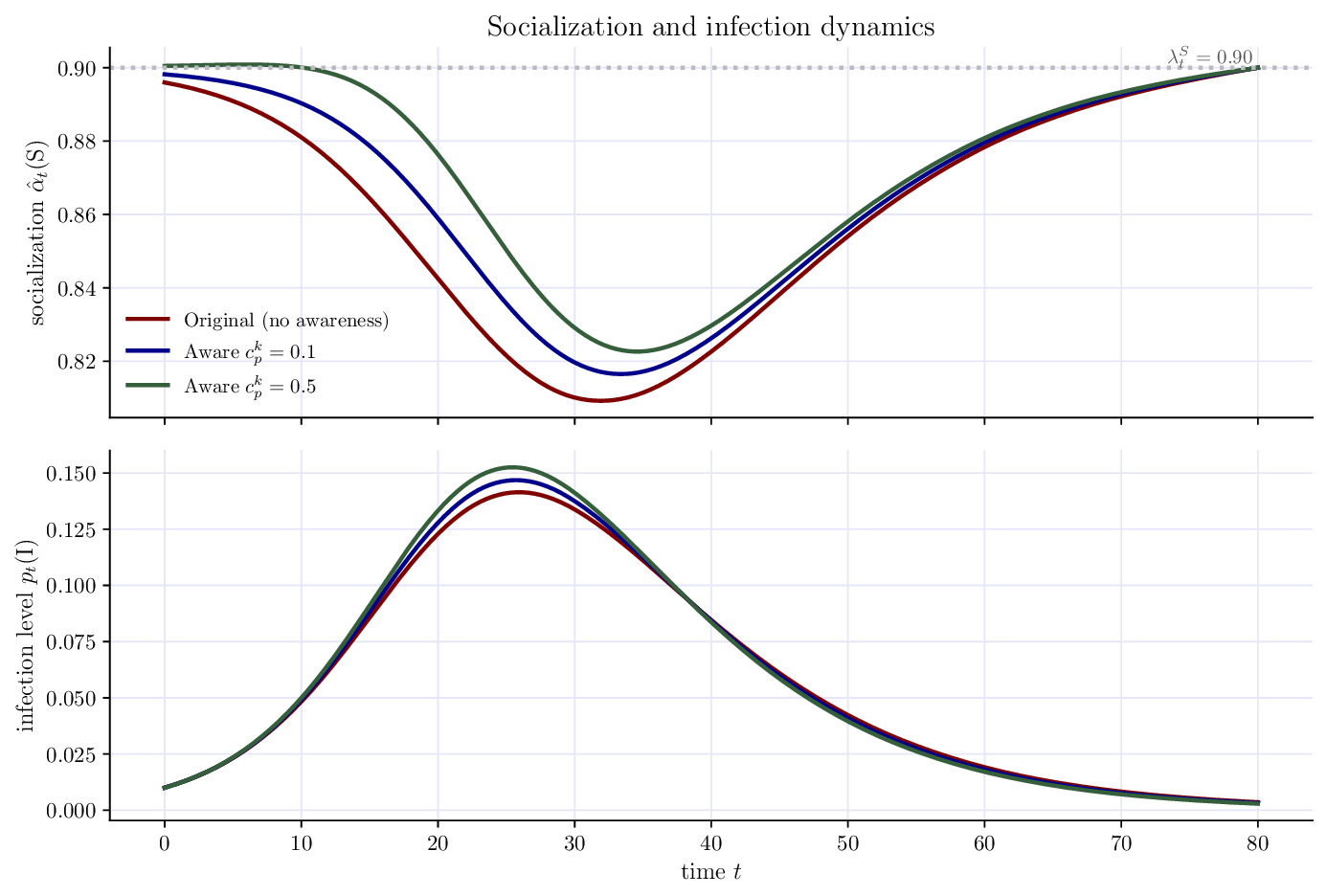}
    \caption{\textbf{Single Population Simulation:} comparison of equilibrium socialization levels $\hat{\alpha}(\rmS)$ and infection proportions $p_t(\rmI)$ across population-awareness coefficients.}
    \label{fig:p3}
\end{figure}

Figure~\ref{fig:p4} and Figure~\ref{fig:p4_infection} report the equilibrium results over a grid of different social-distancing guidelines for susceptible and infected individuals, $(\lambda^{\rmS},\lambda^{\rmI}) \in [0.3,0.9]^2$. In Figure~\ref{fig:p4}, the left panel displays the equilibrium vaccination jump time and shows that, as the social-distancing guideline becomes more permissive (that is, as $\lambda^{\rmS}$ and $\lambda^{\rmI}$ increase), individuals perceive a higher infection risk and therefore, continue vaccinating for a longer time. The right panel presents a heatmap of the minimum susceptible socialization level over $[0,T]$, namely the most protective socialization level adopted over the time horizon ($\min_t \alpha_t(\rmS)$). We see that, $\min_t \alpha_t(\rmS)$ is chosen closer to $\lambda^{\rmS}$, but decreases to lower when the social-distancing guideline for infected individuals becomes more permissive (i.e. when $\lambda^{\rmI}$ is higher). This emphasizes the protective behavior of susceptible individuals when they perceive themselves at further risk.

Figure~\ref{fig:p4_infection} shows the infection peak time in the left panel and the infection peak proportion in the right panel. Interestingly, both the peak time and proportion appear to be driven primarily by the social-distancing guidelines given to infected individuals. For example, when $\lambda^{\rmI}$ is around $0.4$, infection suppression occurs for nearly all values of $\lambda^{\rmS}$. In contrast, when $\lambda^{\rmI}$ is around $0.7$, the government must also choose $\lambda^{\rmS}$ below approximately $0.6$ to achieve infection suppression. Since susceptible individuals constitute the majority of the population in the early stage of an epidemic, imposing strong restrictions on their social activity may be particularly costly from the policy-maker's perspective. These findings therefore underscore the importance of promptly quarantining infected individuals during the initial phase of an epidemic outbreak.

\begin{figure}[ht]
    \centering
    \begin{minipage}[t]{0.48\textwidth}
        \centering
        \includegraphics[width=\textwidth]{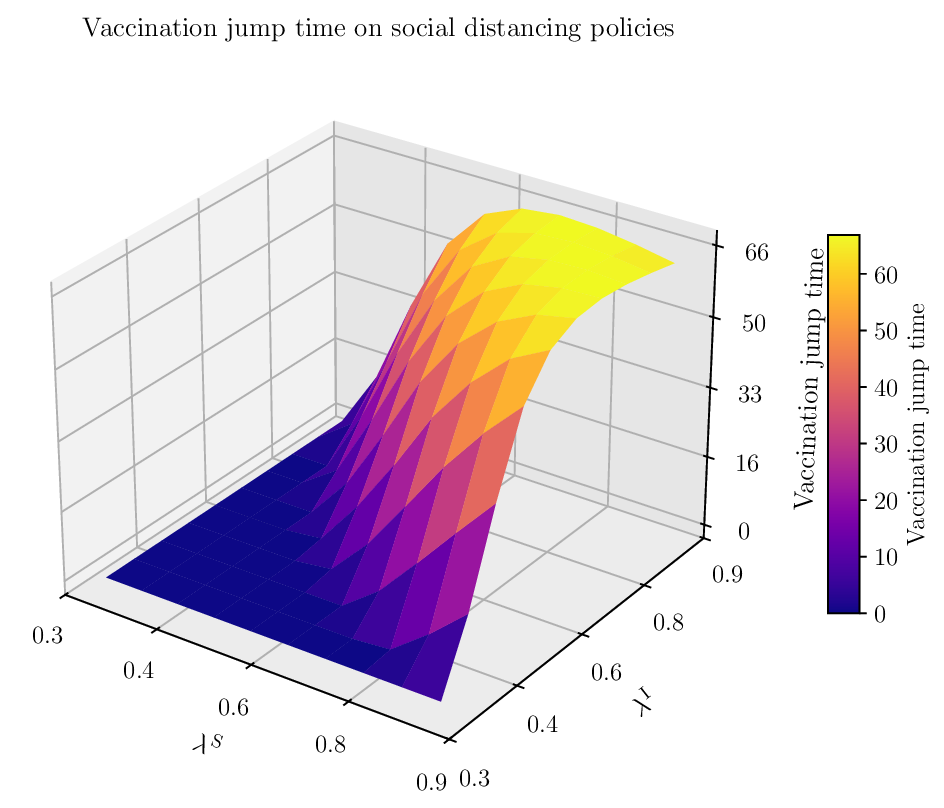}
    \end{minipage}
    \hfill
    \begin{minipage}[t]{0.48\textwidth}
        \centering
        \includegraphics[width=\textwidth]{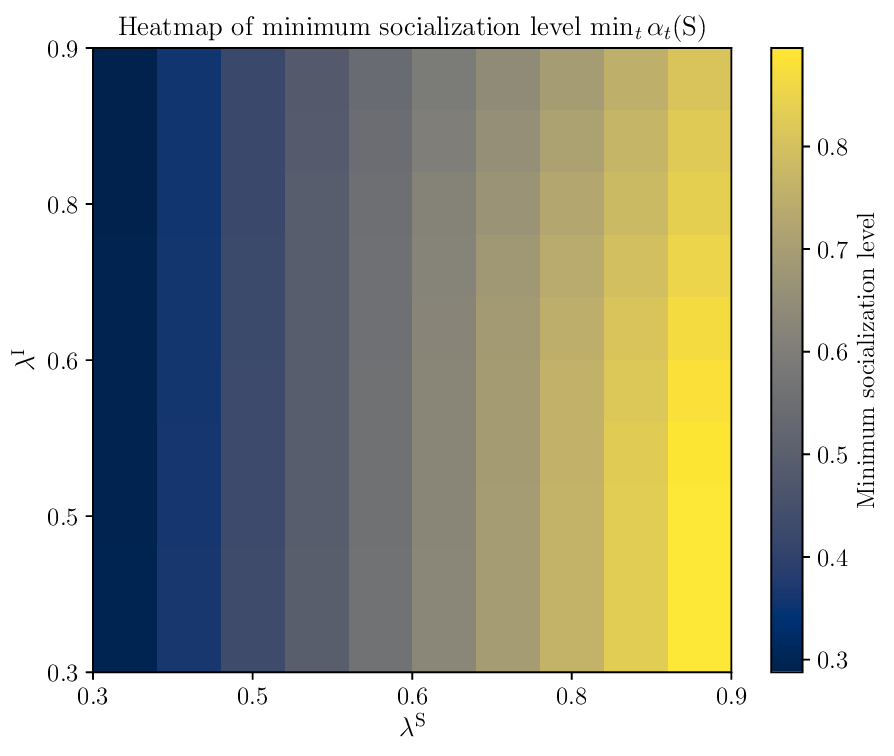}
    \end{minipage}
    \caption{\textbf{Single Population Simulation:} Grid-search comparison over the social-distancing policy space $(\lambda^{\rmS}, \lambda^{\rmI}) \in [0.3,0.9]^2$. \textbf{Left}: the vertical axis shows the equilibrium vaccination jump time. \textbf{Right}: the heatmap shows the minimum socialization level over $[0,T]$, i.e., $\min_t \alpha_t(\rmS)$.}
    \label{fig:p4}
\end{figure}

\begin{figure}[H]
    \centering
    \begin{minipage}[t]{0.48\textwidth}
        \centering
        \includegraphics[width=\textwidth]{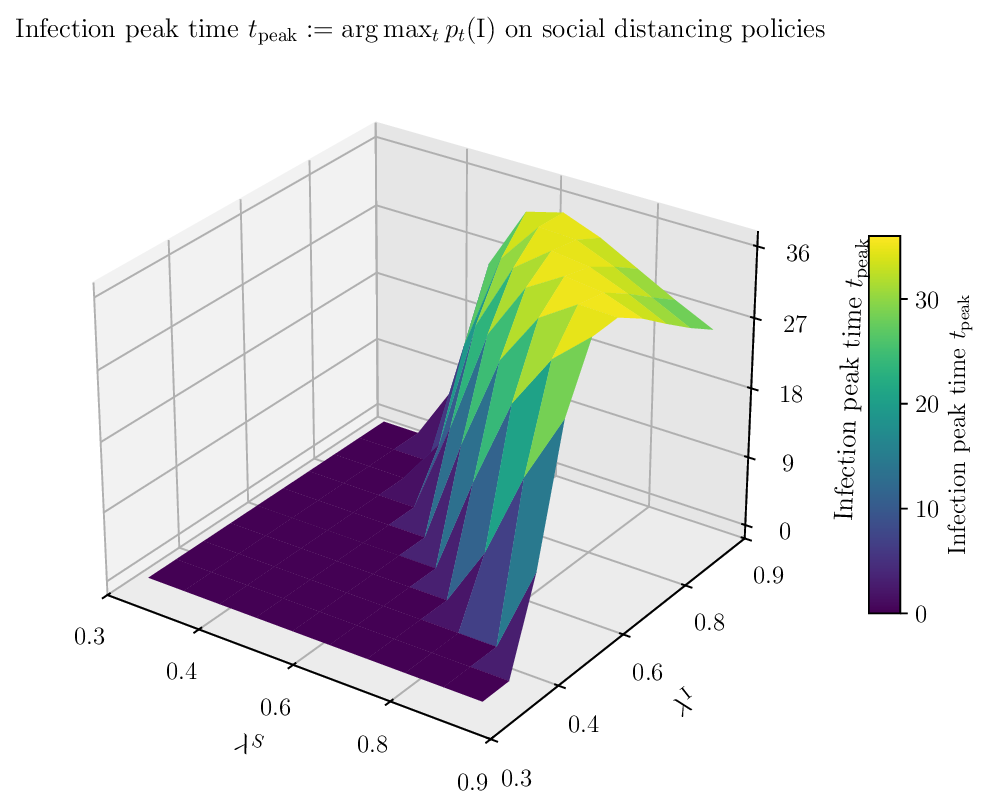}
    \end{minipage}
    \hfill
    \begin{minipage}[t]{0.48\textwidth}
        \centering
        \includegraphics[width=\textwidth]{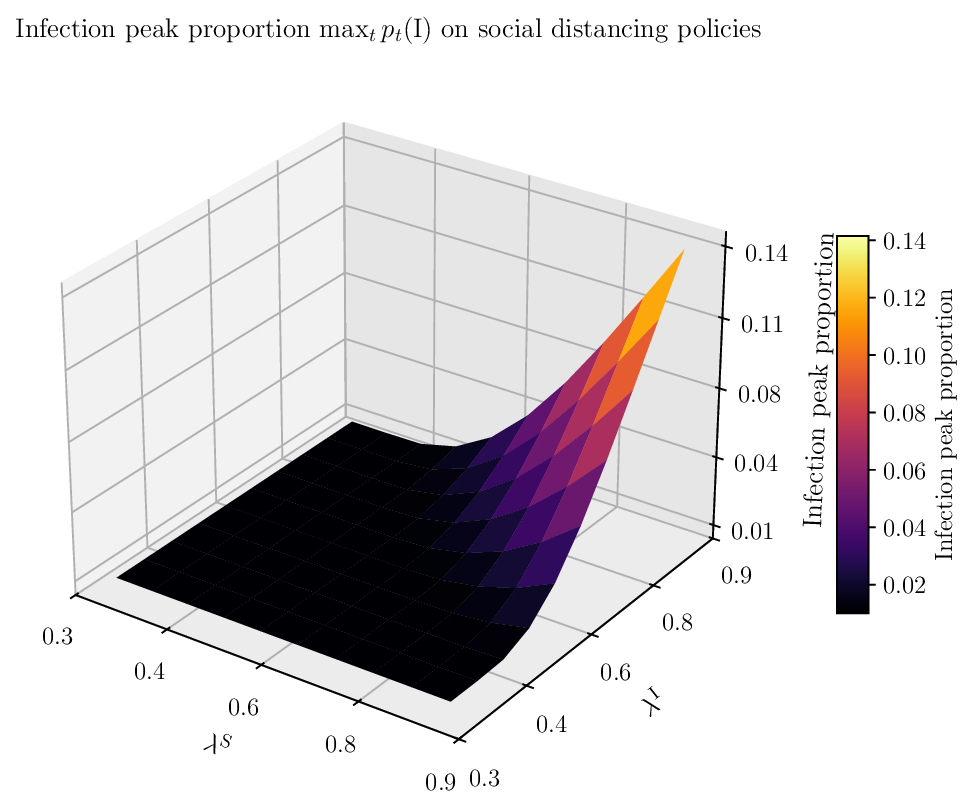}
    \end{minipage}
    \caption{\textbf{Single Population Simulation:} Grid-search comparison over the social-distancing policy space $(\lambda^{\rmS}, \lambda^{\rmI}) \in [0.3,0.9]^2$. Left: the vertical axis shows the infection peak time, i.e., $\arg\max_t p_t(\rmI)$. Right: the vertical axis shows the infection peak proportion, i.e., $\max_t p_t(\rmI)$.}
    \label{fig:p4_infection}
\end{figure}

\subsection{Real-data Informed Applications}
The second numerical experiment is conducted on the baseline model proposed in Section~\ref{sec:mfg_res} with parameters informed by the survey-based calibration in~\cite{liu2025incorporatingauthorityperceptioneconomic}, where a two-wave online survey in Illinois was used to capture population heterogeneity in health-seeking behavior across income and authority-perception groups. In the prior study~\cite{liu2025incorporatingauthorityperceptioneconomic}, the authors analyze a multi-population MFG epidemic model with \textit{quadratic} vaccination costs in which the population is stratified by both economic status (\textit{low}-, \textit{middle}-, and \textit{high-income}) and authority perception (\textit{follower} or \textit{indifferent}), resulting in six distinct groups. For clarity and tractability, we restrict attention to the three economic groups (\textit{low}-, \textit{middle}-, and \textit{high-income}) and aggregate parameters within each income category by taking group-wise averages. The parameter values and contact matrix are reported in Table~\ref{tab:param_real_implication} and Table~\ref{tab:contact_matrix_real_implication} respectively. We emphasize that $c_{\rmI}$ and $c_\nu$ are taken to be group-specific to reflect higher infection and vaccination costs for lower-income individuals due to health insurance access or simply burdens associated with obtaining these services (such as costs of traveling or taking time-off work). For vaccination-related parameters $(c_{\nu}, \kappa)$, since our proposed model introduces a linear vaccination cost, we preserve their relative magnitudes across income groups and apply a rescaling to ensure consistency with the model formulation and the existence of solutions to the coupled FBODE system. Under this calibrated setting, we examine the effects of two social-distancing guidelines and two vaccination-cost regimes on the induced epidemic dynamics.

\begin{table}[htbp]
\centering
\renewcommand{\arraystretch}{1.25}
\begin{tabular}{|l l l|}
\hline
$T = 80$, $\Delta t = 0.016$
& $K = 3$
& $m = (0.3224, 0.3164, 0.3612)$ \\

$\beta = (0.40,\,0.35,\,0.30)$ 
& $\kappa = (0.003,\,0.003,\,0.003)$ 
& $\gamma = (1/7,\,1/7,\,1/7)$ \\

$c_\nu = (0.015,\,0.013,\,0.009)$
& $c_{\mathrm{I}} = (1.05,\,1.00,\,0.80)$
& $c_\lambda = (1.0,\,1.0,\,1.0)$ \\

$\pi_0(\rmS) = (0.99,0.99,0.99)$
& $\pi_0(\rmI) = (0.01,0.01,0.01) $
& $\pi_0(\rmR) = (0,0,0)$ \\

\hline
\end{tabular}
\caption{Parameters of data-informed multi-population experiment. The vectors are presented in order: low-, middle-, high-income.}
\label{tab:param_real_implication}
\end{table}

\begin{table}[htbp]
\centering
\begin{tabular}{c ccc}
\toprule
 & low & middle & high \\
\midrule
low    & 1     & 0.925 & 0.925 \\
middle & 0.925 & 1     & 0.925 \\
high   & 0.925 & 0.925 & 1     \\
\bottomrule
\end{tabular}
\caption{Contact matrix}
\label{tab:contact_matrix_real_implication}
\end{table}

Figure~\ref{fig:p5} compares the baseline social-distancing guideline 
($\lambda_t^{k,\rmS}=\lambda_t^{k,\rmI}=\lambda_t^{k,\rmR}=0.9$),
with an alternative policy that imposes stricter distancing on infected individuals ($\lambda_t^{k,\rmI}=0.6$).
The figure presents the population density flow for states $\rmS$, $\rmI$, and $\rmR$, together with the corresponding equilibrium socialization levels and vaccination rates. From a public health policy perspective, the stricter social-distancing guideline suppresses the outbreak more effectively; as a result, individuals under this policy choose higher equilibrium socialization levels and more interestingly, do not vaccinate. Across income groups, the infection proportion is highest in the low-income class, followed by the middle-income class, and then the high-income class. Moreover, lower-income groups prefer lower socialization levels in exchange for a shorter vaccination period, reflecting the effects of the higher vaccination costs they incur.

Figure~\ref{fig:p6} compares the baseline income-group-specific vaccination-cost structure with a regime of uniformly low vaccination costs. We find that lower vaccination costs encourage all groups to vaccinate for a longer period, with the low-income group exhibiting the largest increase in vaccination duration. As vaccination is sustained for longer durations across all groups, individuals perceive themselves to be better protected and therefore choose higher equilibrium socialization levels (especially after vaccinations). This in turn leads to lower infection proportions.

\begin{figure}[ht]
    \centering
    \includegraphics[width=0.77\linewidth]{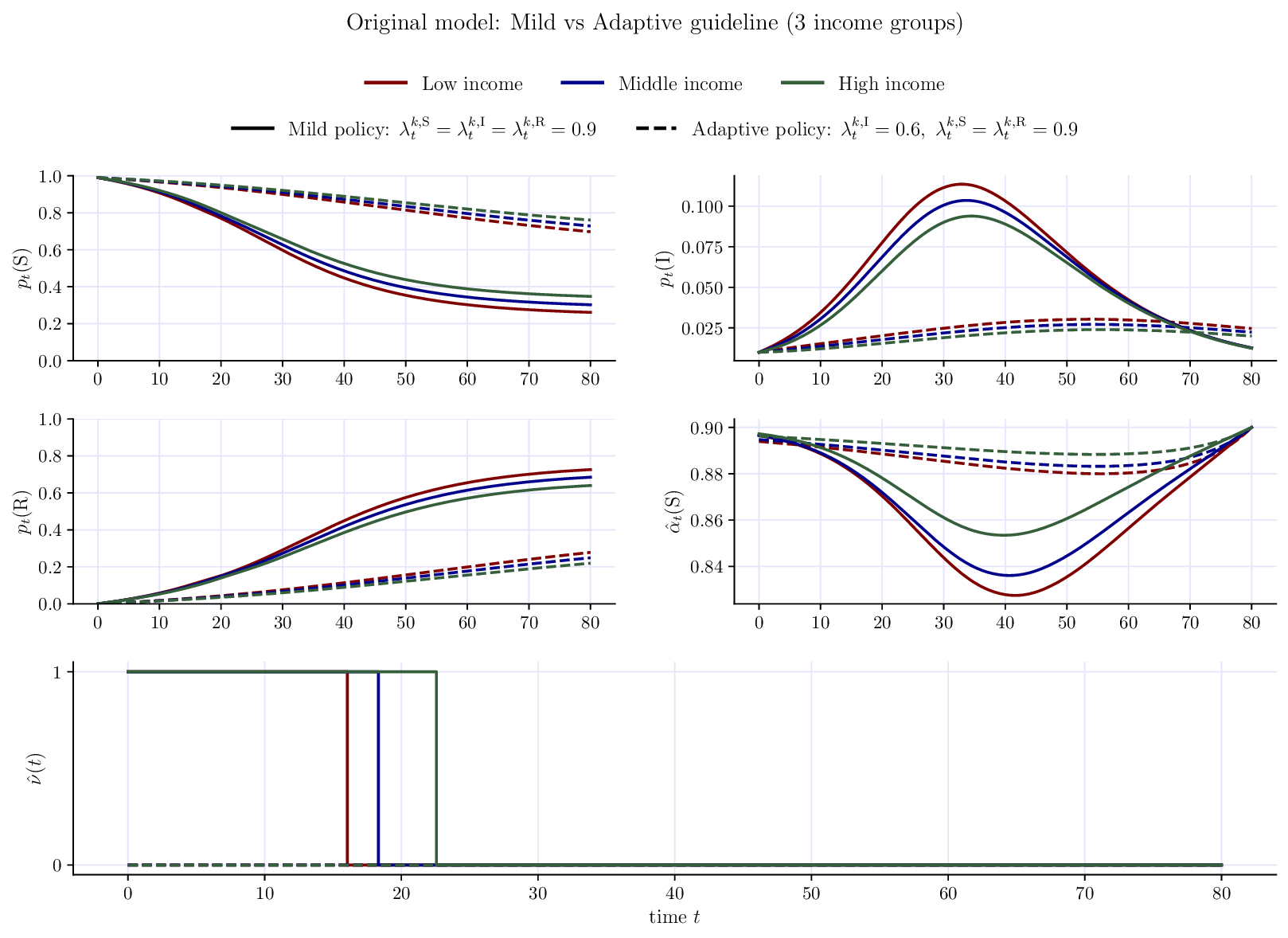}
    \caption{\textbf{Real-data informed application:} Comparison of equilibrium outcomes under a mild social-distancing guideline and an adaptive social-distancing guideline that prescribes a lower socialization level for infected individuals.}
    \label{fig:p5}
\end{figure}

\begin{figure}[H]
    \centering
    \includegraphics[width=0.77\linewidth]{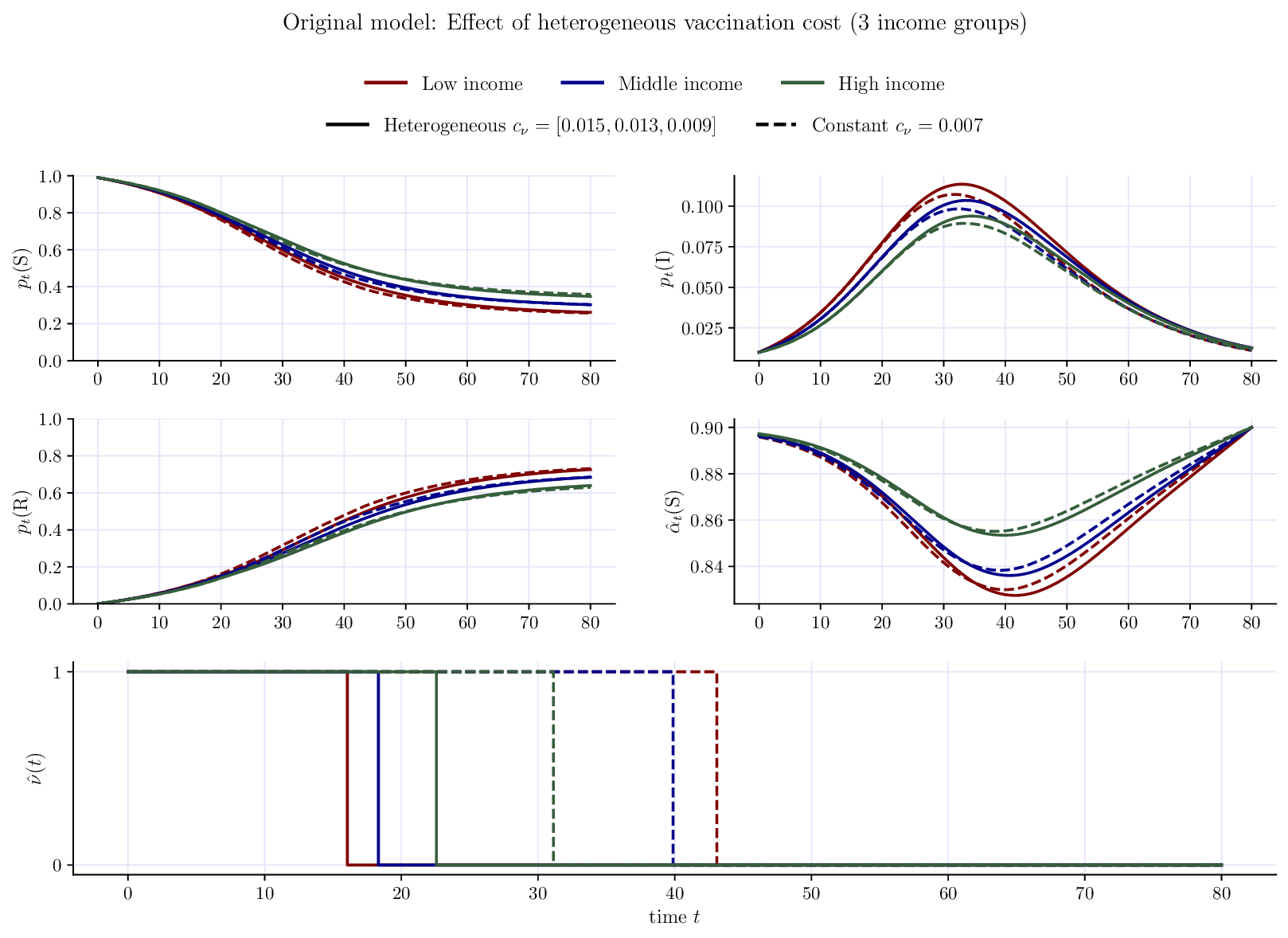}
    \caption{\textbf{Real-data informed application:} Comparison of equilibrium outcomes under the income group based vaccination cost regime and an uniformly low vaccination cost regime for all groups.}
    \label{fig:p6}
\end{figure}

\section{Conclusion}
In this paper, we propose and study an extended MFG framework that captures adaptive human responses during an epidemic outbreak. In particular, we incorporate both socialization and vaccination decisions into the model, together with a \textit{linear vaccination cost structure} that is more realistic in many practical settings. The linear cost structure results in a bang-bang type of vaccination decisions with that is interpreted as individuals choosing to vaccinate until a time decided by model parameters. We further endogenize population-awareness through infection information in the cost function of individuals, yielding a richer behavioral framework in which individuals respond not only to their own health status and direct control costs, but also to the population-level infection burden.

Our theoretical contributions include a characterization of the MFNE through an FBODE system for both the baseline model and its population-awareness extension. For the baseline model, we prove that the equilibrium vaccination rate exhibits an at-most one-jump bang-bang structure: individuals vaccinate at the maximal rate up to a critical time and stop thereafter, with the duration determined by the vaccination cost-to-efficacy ratio. We also establish an existence result for the associated FBODE system. For the population-awareness extension, we derive conditions under which the same at-most one-jump bang-bang structure is preserved. In particular, this remains true when the infection cost remains stronger than the awareness cost coefficient.

Our numerical studies support the theoretical findings on at-most one-jump bang-bang structure of vaccination rates at equilibrium, and provide several biologically meaningful insights. In particular, the experiments highlight the importance of promptly quarantining infected individuals during the early stage of an epidemic outbreak. They also reveal a trade-off between vaccination and socialization decisions, with the balance between these two protective behaviors depending on the weight assigned to population-awareness in the cost. In the real data informed multi-population application with different socioeconomic groups, the simulations further show that socioeconomic differences can affect equilibrium responses and the resulting epidemic trajectories.

Future research directions include extending the framework to incorporate policy-makers decision-making through Stackelberg game formulations, in order to better understand the interaction between epidemic mitigation interventions and strategic behavior in large populations. Another natural direction is to consider alternative forms of population-awareness endogenization, such as delayed or noisy awareness modeled as a separate dynamical process.

\bibliographystyle{plain}
\bibliography{ref}

\appendix
\section*{Appendix}

\section*{Proof of Lemma~\ref{lemma:sub_wellposedness}}
\label{sec:appendix_sub_wellposedness}
Now the vaccination posits as a exogenous parameter of the sub-FBODE system, we recast the sub-FBODE as a fixed-point problem on the density flow. We begin by remarking that the corresponding FBODE system still solves the same multi-population mean field game problem but with socialization level as the single control. Also by separability of the socialization and vaccination levels in Hamiltonian, the expression of the new socialization under equilibrium will not change. Therefore, the inducing FBODE system is~\eqref{eq:fbode} with $\hat{\nu}$ replaced by $\tilde{\nu}$, and we provide its well-posedness below.

In particular, we construct a fixed-point operator associated with the FBODE and show that it is a contraction when the terminal horizon $T$ is sufficiently small. First we note that both $u_t(\rmI)$ and $u_t(\rmR)$ are explicitly solvable, where $$u_t^{k}(\rmI) = \frac{c_{\rmI}^{k}}{\gamma^{k}}(1-e^{-\gamma^{k}(T-t)}) \quad \text{ and }\quad u_t^k(\rmR) = 0 \quad , t \in [0,T].$$ This simplifies the derivation of bounds for value functions, as specified later. We consider the state density flow $p := (p_t^{k}(\rmS),p_t^{k}(\rmI),p_t^{k}(\rmR))_{k \in [K], t \in [0,T]}$ and value function $u := (u_t^{k}(\rmS),u_t^{k}(\rmI),u_t^{k}(\rmR))_{k \in [K], t \in [0,T]}$ as two vector processes. Using the forward and backward components of the FBODE, we first fix $p$ flow and solve the backward HJB equations to obtain $u$; and by fixing these $u$'s, we solve the forward KFP equations to obtain $\tilde{p}$. This procedure is described as $p \mapsto h_1(p)=u \mapsto h_2(u) = \tilde{p}$. Let $h:=h_2 \circ h_1$, we want to show $h$ is a contraction mapping i.e. for $p^{1} = (p^{1,k}_{t}(e))_{t\in[0,T],k\in[K],e\in E}$ and $p^{2} = (p^{2,k}_{t}(e))_{t\in[0,T],k\in[K],e\in E}$, 
    $$\left\|h\left(p^1\right)-h\left(p^2\right)\right\|_T \leq C \left\|p^1-p^2\right\|_T$$ where $C<1$ and $\|x\|_T:=\sup_{t\in[0,T]}\|x_t\|_2$.
First define $\mathcal{M}:=C([0,T];\Delta_E^K)$ equipped with the supremum norm $\|\cdot\|_{T}$. Since $\Delta_E^K$ is closed in
$\mathbb R^{3K}$, the space $\mathcal{M}$ is a closed subset of
$C([0,T];\mathbb R^{3K})$, hence complete. There is $h(\mathcal{M}) \subset \mathcal{M}$.
We introduce the notation that will be used throughout the remainder of the proof until the existence theorem~\ref{thm:fbode_wellposedness}.
Define
$$
\underline{\gamma}:=\min_k\gamma^k,\quad \bar{\gamma}:=\max_k\gamma^k,
$$
$$
\underline{c}:=\min_k\{c_\nu^k,c_{\rmI}^k,c_\lambda^k\},\quad
\bar{c}:=\max_k\{c_\nu^k,c_{\rmI}^k,c_\lambda^k\},\quad
\bar{\beta}:=\max_k\beta^k,\quad \bar{\kappa}:=\max_k\kappa^k.
$$
Let
$$
\bar W:=\max_{k,l\in[K]} w(k,l)m^l,\quad
\bar Z:=\max_{k\in[K]}\sum_{l\in[K]}w(k,l)m^l.
$$

\noindent\textbf{Step 0: Inequality tools. }
    We state and prove (if necessary) some boundedness results that will be used in the proof context. For aggregate function, consider $Z_{t}^{1,k}$ and $Z_{t}^{2,k}$, we have
        \begin{equation}
            \label{eq:bdd_delta_Z}
            \left|\Delta Z_{t}^{k}\right| = \left|Z_{t}^{1,k} - Z_{t}^{2,k}\right| = \Big| \sum_{l \in[K]} w(k, l) \lambda_t^{l, \rmI} m^l \Delta p_t^l(\rmI) \Big|\leq\bar{W} \sum_{l \in [K]}\left|\Delta p_t^{l}(\rmI)\right|.
        \end{equation} We also observed the boundedness of aggregates $\left|Z_t^k\right| \leq \sum_{l \in[K]} w(k, l) m^l$ by using the fact that state density flow is bounded by 1.
    The value function is bounded using the cost functional form, that for $t \in [0,T], k \in [K], e \in E$, $\left|u_t^k(e)\right| \leq 2\bar{c} T =: \bar{U}$.
    We assume $|\hat{\alpha}_{t}^{k}(e)|\leq 1$ for simplicity for all $t \in [0,T], k \in [K], e \in E$. Consider $\hat{\alpha}_t^{1,k}$ and $\hat{\alpha}_t^{2,k}$, by adding and subtracting the cross-terms, it follows
        \begin{equation}
            \label{eq:bdd_delta_alpha}
            \begin{aligned}
|\Delta \hat{\alpha}_t^k(\rmS)| \leq & \frac{\beta^k}{2 c_\lambda^k}\big(|Z_t^{1, k}||\Delta u_t^k(\rmS)|+|u_t^{2, k}(\rmS)||\Delta Z_t^k|+|Z_t^{1, k}||\Delta u_t^k(\rmI)|+|u_t^{1, k}(\rmI)||\Delta Z_t^k|\big) \\
\leq & \frac{\bar{\beta}}{\ubarc} \bar{U}|\Delta Z_t^k|+\frac{\bar{\beta}}{2 \ubarc} \bar{Z}|\Delta u_t^k(\rmS)|+\frac{\bar{\beta}}{2 \ubarc} \bar{Z}|\Delta u_t^k(\rmI)| \\
\leq & \frac{\bar{\beta}}{\ubarc} \bar{U} \bar{W} \sum_{l \in[K]}|\Delta p_t^l(\rmI)|+\frac{\bar{\beta}}{2 \ubarc} \bar{Z}|\Delta u_t^k(\rmS)|.
\end{aligned}
        \end{equation}
    
    \noindent\textbf{Step 1: Contraction boundedness of $h_1$.}
    We start with two processes $p^{1} = (p^{1,k}_{t}(e))_{t\in[0,T],k\in[K],e\in E}$ and $p^{2} = (p^{2,k}_{t}(e))_{t\in[0,T],l\in[K],e\in E}$, and consider $\Delta p = p^1 - p^2$. For state $\rmS$, by chain rule and plugging in the dynamics of value function and finally by applying Young's inequality, we have 
$$\begin{aligned}
        &\hskip-2mm\frac{d}{d t}\|\Delta u_t(\rmS)\|^2 =2 \Delta u_t(\rmS) \cdot \Delta \dot{u}_t(\rmS) \\
        &=\sum_k 2\Delta u_t(\rmS)\Big[\beta^k\Big\{\Delta \hat{\alpha}_t^k(\rmS) Z_t^{1, k}\big(u_t^{1, k}(\rmS)-u_t^{1, k}(\rmI)\big)+\hat{\alpha}_t^{2,k}(\rmS)\Delta Z_t^{k}\big(u_t^{1, k}(\rmS)-u_t^{1, k}(\rmI)\big)\\ &\hspace{1cm}+\hat{\alpha}_t^{2,k}(\rmS)Z_t^{2,k}\big(\Delta u_t^{k}(\rmS)-\Delta u_t^{k}(\rmI)\big)\Big\} - c_\lambda^k \Delta \hat{\alpha}_t^k(\rmS)\big(\hat{\alpha}_t^{1, k}(\rmS)+\hat{\alpha}_t^{2, k}(\rmS)-2 \lambda_t^{k, \rmS}\big) \\ & \hspace{1cm}+ \big((\kappa^{k}u_{t}^{1,k}(\rmS)-c_{\nu}^{k})\mathds{1}_{t < t_1^k} - \big(\kappa^{k}u_{t}^{2,k}(\rmS)-c_{\nu}^{k}\big)\mathds{1}_{t < t_1^k}\big) \Big]\\
        &\leq \sum_{k\in[K]} 2 \Delta u^k_t(\rmS)\Big\{ [2 \bar{\beta} \bar{Z} \bar{U}+2\bar{c}]|\Delta \hat{\alpha}_t^k(\rmS)|+(\bar{\beta} \bar{Z}+\bar{\kappa})|\Delta u_t^k(\rmS)|+ 2\bar{U}|\Delta Z_t^{k}|\Big\} \\
        &\leq \sum_{k\in[K]} 2|\Delta u^k_t(\rmS)|\Big\{\big([2 \bar{\beta} \bar{Z} \bar{U}+2\bar{c}] \bar{U} \bar{W} \frac{\bar{\beta}}{\ubar{c}} + 2\bar{U}\bar{W}\big) \sum_{l \in[K]}|\Delta p_t^l(\rmI)|\\&\hspace{1cm}+\big(\bar{\kappa}+[2 \bar{\beta} \bar{Z} \bar{U}+2\bar{c}]\frac{\bar{\beta}}{2\ubar{c}}\bar{Z}+\bar{\beta}\bar{Z}\big) |\Delta u_t^k(\rmS)|\Big\}
        \\&\leq d_1^{u}\sum_{k\in[K]}|\Delta u_t^{k}(\rmS)|^2 + d_2^{u}\sum_{k\in[K]}|\Delta p_t^{k}(\rmI)|^2,
\end{aligned}$$
where $d_1^{u} := [2 \bar{\beta} \bar{Z} \bar{U}+2\bar{c}]\frac{\bar{\beta}}{\ubarc}\bar{U}\bar{W}+2\bar{U}\bar{W}+2\bar{\kappa}+2\bar{\beta}\bar{Z}+[2 \bar{\beta} \bar{Z} \bar{U}+2\bar{c}]\bar{Z}\frac{\bar{\beta}}{\ubarc}$ and $d_2^u := \big([2 \bar{\beta} \bar{Z} \bar{U}+2\bar{c}] \bar{U} \bar{W} \frac{\bar{\beta}}{\ubar{c}} + 2\bar{U}\bar{W}\big)K$. The first and second inequality used the results from step 0 and the third inequality used Young's inequality. For the value functions on state I and R, they have explicit solutions as stated above, so we trivially have $\frac{d}{d t}\left\|\Delta u_t(\rmI)\right\|^2 = \frac{d}{d t}\left\|\Delta u_t(\rmR)\right\|^2 = 0$. Combine the three states to have
$$\frac{d}{d t}\left\|\Delta u_t\right\|^2 = \frac{d}{d t}\left(\left\|\Delta u_t(\rmS)\right\|^2+\left\|\Delta u_t(\rmI)\right\|^2+\left\|\Delta u_t(\rmR)\right\|^2\right)\leq A_1\|\Delta u_t\|^2 + A_2\|\Delta p_t\|^2.$$
Here $A_1 := d_1^{u}$ and $A_2 := d_2^u$. Then apply Gr\"{o}nwall's inequality to have
    $$\left\|\Delta u_t\right\|^2 \leq A_2 \int_t^T e^{A_1(t-s)}\left\|\Delta p_s\right\|^2 d s \leq A_2 e^{A_{1}T} \int_0^{T}\|\Delta p_s\|^2ds.$$

\noindent\textbf{Step 2: Contraction boundedness of $h_2$.}
For the two processes $\tilde{p}^1 := h(p^1)$ and $\tilde{p}^2 := h(p^2)$, denote $\Delta \tilde{p} = \tilde{p}^1 - \tilde{p}^2$. First consider state $\rmS$, by chain rule and plugging in the dynamics of state density flow and finally by applying Young's inequality, we have
$$\begin{aligned}
        \frac{d}{dt} \|\Delta \tilde{p}_t(\rmS)\|^2 &= 2\Delta \tilde{p}_t(\rmS)\cdot\Delta \dot{\tilde{p}}_t(\rmS) \\ & \hspace{-1.8cm}=
\sum_{k\in[K]} 2 \Delta \tilde{p}_t^k(\rmS) \Big[\sum_{l \in[K]} w(k, l) \lambda_t^{l, \rmI} m^l\big(-\beta^k \hat{\alpha}_t^{1,k}(\rmS) \tilde{p}_t^{1,l}(\rmI) \tilde{p}_t^{1,k}(\rmS)+\beta^{k} \hat{\alpha}_t^{2, k}(\rmS) \tilde{p}_t^{2,l}(\rmI) \tilde{p}_t^{2, k}(\rmS)\big) -\kappa^k \tilde{\nu}_t^{k} \Delta \tilde{p}_t^k(\rmS)\Big] \\ &\hspace{-1.8cm}\leq 
\sum_{k\in[K]} 2|\Delta \tilde{p}_t^k(\rmS)|\Big[\bar{W} \bar{\beta} \sum_{l\in[K]}\big(|\Delta \tilde{p}_t^k(\rmS)|+|\Delta \tilde{p}_t^l(\rmI)|+|\Delta \hat{\alpha}_t^k(\rmS)|\big)+\bar{\kappa} |\Delta \tilde{p}_t^k(\rmS)|\Big]
\\ &\hspace{-1.8cm}\leq 
\sum_{k\in[K]} 2|\Delta \tilde{p}_t^k(\rmS)|\Big[\big(K\bar{W}\bar{\beta}+\bar{\kappa}\big)|\Delta \tilde{p}_t^k(\rmS)|+\big(K\bar{W}^2\bar{U}\frac{\bar{\beta}^2}{\ubarc}+\bar{W}\bar{\beta}\big)\sum_{l \in [K]}|\Delta \tilde{p}_t^l(\rmI)| + K\bar{W}\bar{Z}\frac{\bar{\beta}^2}{2\ubarc}|\Delta u_t^k(\rmS)|\Big]
\\ &\hspace{-1.8cm} \leq \sum_{k\in[K]} \Big(3K\bar{W}\bar{\beta}+2\bar{\kappa}+K^2\bar{W}^2\bar{U}\frac{\bar{\beta}^2}{\ubarc}+K\bar{W}\bar{Z}\frac{\bar{\beta}^2}{2\ubarc}\Big)|\Delta \tilde{p}_t^k(\rmS)|^2 \\ & \hspace{2cm}+\big(K^2\bar{W}^2\bar{U}\frac{\bar{\beta}^2}{\ubar{c}}+K\bar{W}\bar{\beta}\big)|\Delta \tilde{p}_t^k(\rmI)|^2 + K \bar{W} \bar{Z} \frac{\bar{\beta}^2}{2 \ubar{c}}|\Delta u_t^{k}(\rmS)|^2,
    \end{aligned}$$
For state I and R, it similarly follows
$$\begin{aligned}
    \frac{d}{dt} \left\|\Delta \tilde{p}_t(\rmI)\right\|^2 &= 2\Delta \tilde{p}_t(\rmI)\cdot\Delta \dot{\tilde{p}}_t(\rmI) \\ &\leq \sum_{k\in[K]} \big(2K\bar{W} \bar{\beta}+2K^2 \bar{W}^2 \bar{\beta}^2 \frac{\bar{U}}{\ubar{c}}+K\bar{W} \bar{\beta}^2 \frac{\bar{Z}}{2\ubar{c}}+2 \bar{\gamma}\big)\left|\Delta \tilde{p}_t^k(\rmI)\right|^2 \\ &\quad\quad\quad\quad\quad + K\bar{W} \bar{\beta} \left|\Delta \tilde{p}_t^k(\rmS)\right|^2 + K\bar{W} \bar{\beta}^2 \frac{\bar{Z}}{2\ubar{c}} \left|\Delta u_t^k(\rmS)\right|^2,  \\ \frac{d}{dt} \left\|\Delta \tilde{p}_t(\rmR)\right\|^2 &= 2\Delta \tilde{p}_t(\rmR)\cdot\Delta \dot{\tilde{p}}_t(\rmR) \\ &\leq \sum_{k\in[K]}(\bar{\gamma}+\bar{\kappa})\left|\Delta \tilde{p}_t^k(\rmR)\right|^2+ \bar{\gamma} \left|\Delta \tilde{p}_t^k(\rmI)\right|^2+\bar{\kappa}\left|\Delta \tilde{p}_t^k(\rmS)\right|^2.
\end{aligned}$$
We combine the inequalities for the three states to have
$$\begin{aligned}
    \frac{d}{d t}\left\|\Delta \tilde{p}_t\right\|^2&=\frac{d}{d t}\big(\left\|\Delta \tilde{p}_t(\rmS)\right\|^2+\left\|\Delta \tilde{p}_t(\rmI)\right\|^2+\left\|\Delta \tilde{p}_t(\rmR)\right\|^2\big)\\ &\leq d_1^p \sum_{k\in[K]}|\Delta \tilde{p}_t^k(\rmS)|^2+d_2^p \sum_{k\in[K]}|\Delta \tilde{p}_t^k(\rmI)|^2 \\ &\quad\quad\quad+d_3^p \sum_{k\in[K]}|\Delta \tilde{p}_t^k(\rmR)|^2+d_4^p \sum_{k\in[K]}|\Delta u_t^k(\rmS)|^2 \\ & \leq B_1 \|\Delta \tilde{p}_t\|^2 + B_2 \|\Delta u_t\|^2,
\end{aligned}$$
where $d_1^p:=4K\bar{W}\bar{\beta}+3\bar{\kappa}+K^2\bar{W}^2\bar{U}\frac{\bar{\beta}^2}{\ubarc}+K\bar{W}\bar{Z}\frac{\bar{\beta}^2}{2\ubarc}$, $d_2^p:= 3K^2\bar{W}^2\bar{U}\frac{\bar{\beta}^2}{\ubar{c}}+3K\bar{W}\bar{\beta}+K\bar{W} \bar{\beta}^2 \frac{\bar{Z}}{2\ubar{c}}+3 \bar{\gamma}$, $d_3^p:=\bar{\gamma}+\bar{\kappa}$, $d_4^p:=K \bar{W} \bar{Z} \frac{\bar{\beta}^2}{\ubar{c}}$; and $B_1 = \max \left\{d_1^p, d_2^p, d_3^p\right\}$, $B_2 = d_4^{p}$. Apply Gr\"{o}nwall's inequality to have $$\left\|\Delta \tilde{p}_t\right\|^2 \leq B_2 \exp \left(B_1 T\right) \int_0^T\left\|\Delta u_s\right\|^2 d s$$ Combine this with the result in step 1 to derive
$$\left\|\Delta \tilde{p}_t\right\|^2 \leq B_2 \exp \left(B_1 T\right) \int_0^T A_2 \exp(A_{1}T) \int_0^T\left\|\Delta p_s\right\|^2 d s d v$$
Take supremum over $t$ on both sides and simplify to conclude
$$\sup_{t\in [0,T]}\|\Delta\tilde{p}_t\|^2 \leq A_2 B_2 \exp(A_{1}T) \exp \left(B_1 T\right) T^2 \sup _t\|\Delta p_t\|^2 =: C_1 \sup _t\|\Delta p_t\|^2$$
Under small time horizon $T$, there is $C_1 = A_2 B_2 \exp \left((A_1 + B_1) T\right) T^2 < 1$ and $h$ is a contraction mapping. Banach Fixed Point Theorem applies, and this completes the proof for the sub-FBODE well-posedness.

\section*{Proof of Lemma~\ref{lemma:sub_continuity}}
\label{sec:appendix_sub_continuity}
We first fix a jump time $t_1$ and under the short-time condition uniquely obtain
a solution pair $(u^{t_1},p^{t_1})$ to the FBODE system.
Let $t_1^n\to t_1$ and denote the corresponding solutions by
$(u^n,p^n)$ and $(u,p)$ respectively. For $t \in [0,T]$, $k \in [K]$ and $e \in E$, we further let $\Delta p_t^{k}(e) = p^{n,k}_{t}(e) - p^{k}_{t}(e)$, $\Delta u_t^{k}(e) = u^{n,k}_{t}(e) - u^{k}_{t}(e)$, $\Delta \hat{\alpha}^{k}_t(e) = \hat{\alpha}^{n,k}(e) - \hat{\alpha}^{k}(e)$ and $\Delta Z_{t}^{k} = Z_{t}^{n,k} - Z_{t}^{k}$.

\medskip
\noindent\textbf{Step 0: (more) Inequality tools.}
For each $k$,
$$
|\Delta Z_t^k|
=
\Big|
\sum_{\ell\in[K]} w(k,l)\lambda_t^{l,\rmI} p_t^{l,n}(\rmI)m^l
-
\sum_{\ell\in[K]} w(k,l)\lambda_t^{l,\rmI} p_t^{l}(\rmI)m^l
\Big|
\leq \bar{W} \sum_{l \in[K]} \bigl| \Delta p_t^{l}(\rmI)\bigr|
$$
By the definition of $\hat\alpha_t^k$, we also obtain
$\bigl|\hat\alpha_t^{n,k}-\hat\alpha_t^{k}\bigr|
\leq \frac{\bar{\beta}\bar{U}\bar{W}}{\ubar{c}} \sum_{l \in [K]}|\Delta p_t^l(\rmI)|
+ \frac{\bar{\beta}\bar{Z}}{2\ubar{c}} |\Delta u_t^k(\rmS)|$.

\medskip
\noindent\textbf{Step 1: Estimate on the forward equations.}
Denote the jump time difference interval as $I^{n,k}:=[\min\{t_1^{n, k}, t_1^k\},$ $\max \{t_1^{n, k}, t_1^k\}]$. For the density flow on susceptible state, we apply Chain rule and Young's inequality to have

\begin{equation*}
    \begin{aligned}
\frac{d}{dt}\,\|\Delta p_t(\rmS)\|^2 &=  2\Delta p_t(\rmS) \Delta \dot{p}_t(\rmS) \\
&\hspace{-1.8cm}\leq
\sum_{k} 2 \Delta p_t^k(\rmS)\Bigl[
\bar\beta \Delta\hat\alpha_t^k(\rmS)\,Z_t^{n,k}\,p_t^{n,k}(\rmS)
+ \bar\beta \hat{\alpha}_t^k(\rmS)\,\Delta Z_t^k\,p_t^{n,k}(\rmS) \\
&\qquad\qquad\qquad\qquad\;
+ \bar\beta \hat{\alpha}_t^k(\rmS)\,Z_t^k\,\Delta p_t^k(\rmS)
+ \bar\kappa\Delta\tilde{\nu}_t^k\,p_t^{n,k}(\rmS)
+ \bar\kappa \tilde{\nu}_t^k\,\Delta p_t^k(\rmS)
\Bigr] \\[0.4em]
&\hspace{-1.8cm}\leq
\sum_{k} 2\bigl|\Delta p_t^k(\rmS)\bigr|
\Bigl[
\bar\beta\bar Z\bigl|\Delta\hat\alpha_t^k(\rmS)\bigr|
+ \bar\beta\bigl|\Delta Z_t^k\bigr|
+ \bar\beta\bar Z\,\bigl|\Delta p_t^k(\rmS)\bigr| 
+ \bar\kappa \mathds{1}_{\{t\in I^{n,k}\}}
+ \bar\kappa\,\bigl|\Delta p_t^k(\rmS)\bigr|
\Bigr] \\[0.4em]
&\hspace{-1.8cm}\leq
\sum_{k} 2\bigl|\Delta p_t^k(\rmS)\bigr|
\Bigl[
(\bar\beta\bar{Z}+\bar{\kappa})\bigl|\Delta p_t^k(\rmS)\bigr|
+ \bigl(K\tfrac{\bar{\beta}^2}{\ubar{c}}\bar{Z}\bar{U}\bar{W} + K\bar{\beta}\bar{W}\bigr)\bigl|\Delta p_t^k(\rmI)\bigr| 
+ \frac{\bar\beta^2 \bar{Z}^2}{\ubarc}\,\bigl|\Delta u_t^k(\rmS)\bigr|
+ \bar\kappa \,\mathds{1}_{\{t\in I^{n,k}\}}
\Bigr] \\[0.4em]
&\hspace{-1.8cm}\leq
\Bigl(
2\bar\beta\bar{Z} + 3\bar{\kappa}
+ K\frac{\bar\beta^2}{\ubar{c}}\bar{Z}\bar{U}\bar{W}
+ K\bar\beta\bar{W}
+ \frac{\bar\beta^2\bar{Z}^2}{\ubar{c}}
\Bigr)\sum_{k}\bigl|\Delta p_t^k(\rmS)\bigr|^2 \\
&\hspace{-1.8cm}\qquad
+
\Bigl(
K\frac{\bar\beta^2}{\ubar{c}}\bar{Z}\bar{U}\bar{W}
+ K\bar\beta\bar W
\Bigr)\sum_{k}\bigl|\Delta p_t^k(\rmI)\bigr|^2
+
\frac{\bar\beta^2\bar{Z}^2}{\ubar{c}}\sum_{k}\bigl|\Delta u_t^k(\rmS)\bigr|^2
+
\bar{\kappa} \sum_{k}\mathds{1}_{\{t\in I^{n,k}\}}.
\end{aligned}
\end{equation*}
For state I and R, we similarly have

\begin{align*}
\frac{d}{dt}\|\Delta p_t(\rmI)\|^2
&\leq
\bigl(2K \bar W \bar\beta 
+ 2K^2 \bar{W}^2 \bar{\beta}^2 \frac{\bar{U}}{c}
+ K\bar W \bar{\beta}^2 \frac{\bar{Z}}{2\ubarc}
+ 2\bar{\gamma}
\bigr)
\sum_{k} |\Delta p_t^k(\rmI)|^2 \\& \hspace{3cm}+ K\bar{W} \bar{\beta} \sum_{k} |\Delta p_t^k(\rmS)|^2 +  K\bar{W} \bar{\beta}^2 \frac{\bar{Z}}{2\ubar{c}} \sum_{k} |\Delta u_t^k(\rmS)|^2
\\
\frac{d}{dt}\|\Delta p_t(\rmR)\|^2
&\leq
\sum_{k} 2|\Delta p_t^k(\rmR)|
\Bigl[
\bar\gamma|\Delta p_t^k(\rmI)|
+ \bar{\kappa} \mathds{1}_{\{t\in I^{n,k}\}}
+ \bar{\kappa} |\Delta p_t^k(\rmS)|
\Bigr]
\\
&\hspace{-1.2cm}\leq
(\bar{\gamma} + 2\bar\kappa)\sum_k |\Delta p_t^k(\rmR)|^2
+ \bar{\gamma} \sum_k |\Delta p_t^k(\rmI)|^2
+ \bar{\kappa} \sum_k |\Delta p_t^k(\rmS)|^2
+ \bar{\kappa} \sum_k \mathds{1}_{\{t\in I^{n,k}\}}.
\end{align*}
Combine the SIR states to have
$$
\frac{d}{dt}\|\Delta p_t\|^2
\leq
B_1\|\Delta p_t\|^2
+ B_2\|\Delta u_t\|^2
+ 2\bar{\kappa} \sum_k \mathds{1}_{\{t\in I^{n,k}\}},
$$
where $B_1 := \max\{2\bar\beta\bar{Z} + 4\bar{\kappa}
+ K\frac{\bar\beta^2}{\ubar{c}}\bar{Z}\bar{U}\bar{W}
+ 2K\bar\beta\bar{W}
+ \frac{\bar\beta^2\bar{Z}^2}{\ubar{c}}, 
K\frac{\bar\beta^2}{\ubar{c}}\bar{Z}\bar{U}\bar{W}
+ 3K\bar\beta\bar{W}
+ 2K^2 \bar{W}^2 \bar{\beta}^2 \frac{\bar{U}}{c}
+ K\bar W \bar{\beta}^2 \frac{\bar{Z}}{2\ubarc}
+ 3\bar{\gamma},
\bar{\gamma} + 2\bar{\kappa}\}$, $B_2 := \frac{\bar\beta^2\bar{Z}^2}{\ubar{c}} + K\bar{W} \bar{\beta}^2 \frac{\bar{Z}}{2\ubar{c}}$.
By Gr\"{o}nwall's inequality, for any $t \in [0,T]$, we have
$$
\|\Delta p_t\|^2
\le
B_2 e^{B_1T}\int_0^T \|\Delta u_s\|^2 ds
+ 2\bar\kappa e^{B_1T} \sum_k |t_1^{n,k}-t_1^k|.
$$
This yields
$$
\sup_t \|\Delta p_t\|^2
\leq
B_2 e^{B_1T} T \sup_t\|\Delta u_t\|^2
+ 2\bar\kappa e^{B_1T} \sum_k |t_1^{n,k}-t_1^k|.
$$

\medskip
\noindent\textbf{Step 2: Estimate on the backward equations.}
Similarly, for the susceptible value function,
\begin{equation*}
    \begin{aligned}
        \frac{d}{dt}\|\Delta u_t(\rmS)\|^2 &= 2 \Delta u_t(\rmS) \cdot \Delta \dot{u}_t(\rmS) \\
&\leq
\sum_k 2|\Delta u_t^k(\rmS)|
\Bigl[
(2\bar\beta\bar{Z}\bar{U}+2\bar{c})|\Delta\hat\alpha_t^k(\rmS)|
+ (\bar{\beta}\bar{Z}+\bar{\kappa})|\Delta u_t^k(\rmS)| 
+ 2\bar{\beta}\bar{U}|\Delta Z_t^k|
+ (\bar{\kappa}\bar{U}+\bar{c}) \mathds{1}_{\{t\in I^{n,k}\}}
\Bigr] \\
&\leq D_1 \sum_k|\Delta u_t^k(\rmS)|^2 + D_2 \sum_k|\Delta p_t^k(\rmI)|^2 +  (\bar{\kappa}\bar{U}+\bar{c})\sum_k \mathds{1}_{\{t\in I^{n,k}\}},
    \end{aligned}
\end{equation*}
Here $D_1:= \frac{\bar{\beta}^2\bar{Z}^2}{\ubar{c}}\bar{U}+\frac{\bar{\beta}\bar{Z}\bar{c}}{\ubar{c}}+\bar{\beta}\bar{Z}+\bar{\kappa}+2\bar{\beta}^2\bar{U}^2\bar{Z}\bar{W}\frac{K}{\ubar{c}}+2\frac{\bar{\beta}\bar{U}\bar{c}}{\ubar{c}}\bar{W}K+2\bar{\beta}\bar{U}\bar{W}K+\bar{\kappa}\bar{U} +\bar{c}$, $D_2:=2\bar{\beta}^2\bar{U}^2\bar{Z}\bar{W}\frac{K}{\ubar{c}}+2\frac{\bar{\beta}\bar{U}\bar{c}}{\ubar{c}}\bar{W}K+2\bar{\beta}\bar{U}\bar{W}K$. Relax the bound and re-arrange to have
$$
\frac{d}{dt}\|\Delta u_t\|^2 \leq D_1 \|\Delta u_t\|^2 + D_2 \|\Delta p_t\|^2 +  (\bar{\kappa}\bar{U}+\bar{c})\sum_k \mathds{1}_{\{t\in I^{n,k}\}}
$$
Again Gr\"{o}nwall's inequality, for any $t \in [0,T]$ we have
\[
\|\Delta u_t\|^2
\le
D_2 e^{D_1T} \int_0^T \|\Delta p_s\|^2 ds
+ (\bar\kappa\bar{U}+\bar{c})e^{D_1T}
\sum_k |t_1^{n,k}-t_1^{k}|.
\]
Take supremum on left hand side and plug in the result from step 1, we have
\begin{equation*}
    \begin{aligned}
        \sup_t \|\Delta u_t\|^2 &\leq D_2 e^{D_1T}T \sup_s \|\Delta p_s\|^2 + (\bar\kappa\bar{U}+\bar{c})e^{D_1T}
\sum_k |t_1^{n,k}-t_1^{k}| \\ & \leq D_2 e^{D_1T}T \Big(B_2 e^{B_1T} T \sup_t\|\Delta u_t\|^2
+ 2\bar\kappa e^{B_1T} \sum_k |t_1^{n,k}-t_1^k|\Big) \\& \hspace{4cm}+ (\bar\kappa\bar{U}+\bar{c})e^{D_1T}
\sum_k |t_1^{n,k}-t_1^{k}|
    \end{aligned}
\end{equation*}
For small enough $T > 0$ such that $B_2 D_2 e^{(B_1 + D_1)T}T^2 < 1$, we have
$$\sup_t \|\Delta u_t\|^2 \leq \frac{2\bar{\kappa}D_2e^{(B_1+D_1)T}T + (\bar\kappa\bar{U}+\bar{c})e^{D_1T}}{1 - B_2 D_2 e^{(B_1 + D_1)T}T^2}\sum_k |t_1^{n,k}-t_1^{k}| \longrightarrow 0 \quad\text{  as  }\quad n \longrightarrow \infty.$$

\section*{Proof of Lemma~\ref{lemma:sub_one_jump}}
\label{sec:appendix_sub_one_jump}
We follow the proof of the one-jump property established earlier in~\ref{thm:charac_bang_bang} with some modifications.
First note that the value function of the state I and R still admits an explicit
solution $u^{t_1}_t(\rmI)=\frac{c_{\rmI}^k}{\gamma^k}\bigl(1-e^{-\gamma^k(T-t)}\bigr)$, $u^{t_1}_t(\rmR) = 0$, as their parameters are not vaccination-related and $u^{t_1}_t(\rmI)$ is monotonically decreasing on $[0,T]$.
The value function of the susceptible state satisfies the HJB equation
$$
\begin{aligned}
    -\dot u^{t_1,k}_t(\rmS)
&=
\inf_{\alpha\in[0,1]}
\bigl\{
c_\lambda^k(\lambda^{k,\rmS}_t-\alpha)^2
+ c_\nu^k \tilde\nu_t^k 
+ \beta \alpha Z_t^k \bigl(u^{t_1,k}_t(\rmI)-u^{t_1,k}_t(\rmS)\bigr)
+ \kappa^k \tilde\nu_t^k \bigl(u^{t_1,k}_t(\rmR)-u^{t_1,k}_t(\rmS)\bigr)
\bigr\}.
\end{aligned}
$$
By the property of infimum, taking the feasible control $\alpha=\lambda^{k,\rmS}_t$ yields
\begin{equation}
\label{eq:relax_uS_dynamics}
\begin{aligned}
        -\dot u^{t_1}_t(\rmS)
&\leq
c_\nu^k \tilde\nu_t
- \kappa^k \tilde\nu_t u^{t_1}_t(\rmS)
+ \beta^k \lambda^{k,\rmS}_t Z_t^k \bigl(u^{t_1,k}_t(\rmI)-u^{t_1,k}_t(\rmS)\bigr)
\\&\leq
c_\nu \tilde\nu_t
+ \beta^k \lambda^{k,\rmS}_t Z_t \bigl(u^{t_1,k}_t(\rmI)-u^{t_1,k}_t(\rmS)\bigr).
\end{aligned}
\end{equation}
At time $T$, since $u_T(\rmI)=u_T(\rmS)=0$, we obtain $-\dot u^{t_1}_T(\rmS) \leq c_\nu \tilde\nu_T \leq c_\nu < c_{\rmI}.$
Moreover, recall that $-\dot u^{t_1}_T(\rmI)=c_I$. Therefore it follows that
$u^{t_1}_{T-}(\rmS) < u^{t_1}_{T-}(\rmI)$. By continuity of the value functions, we claim that
$u^{t_1}_t(\rmS) < u^{t_1}_t(\rmI)$ for all $t\in(0,T]$.
Suppose by contradiction that there exists $\tau\in[0,T)$ which is the last time
such that $u^{t_1}_\tau(\rmS)=u^{t_1}_\tau(\rmI)$. Then moving from a backward direction, for $u^{t_1}_t(\rmS)$ to hit $u^{t_1}_t(\rmI)$ at time $\tau$, it is necessary that
$\dot u^{t_1}_\tau(\rmS) \leq \dot u^{t_1}_\tau(\rmI)$ by continuity. Note by its explicit solution, $\dot{u}^{t_1}_\tau(\rmI) = -c_{\rmI}^k e^{-\gamma^k(T-\tau)} \leq -c_{\rmI}^k e^{-\gamma^kT}$. On the other hand by~\eqref{eq:relax_uS_dynamics}, we have $\dot{u}_{\tau}(\rmS)\geq -c_{\nu}^k$. Under short time where $T < \min_{k}\frac{1}{\gamma^k} \log\frac{c_{\rmI}^k}{c_{\nu}^k}$, there is contradiction. Therefore we have $u^{t_1,k}_t(\rmS) < u^{t_1,k}_t(\rmI), \forall t\in[0,T), k\in[K].$ The rest of the proof follows a similar logic as theorem~\ref{thm:charac_bang_bang}.

\section*{Proof of Lemma~\ref{lemma:jump_time_continuity}}
\label{sec:appendix_jump_time_continuity}
We begin with a remark for the condition in $\mathcal{S}$ where the rate is required to be negative at the meeting point. This is to exclude the case where $A_0^k < c^k$ and tangent with $c^k$ from below at some time, and it actually stems from the fact that $\dot{u}_t^k(\rmS) < 0$ at jump time implicitly stated in proof for~\ref{lemma:sub_one_jump}. With this case excluded, meeting with $c^k$ is equivalent to $A_0^k > c^k$, and not meeting with $c^k$ is equivalent to $A_0^k < c^k$. We prove the continuity of $\tilde{t}$ on the two cases separately and consider the boundary point where $A_0^k = c^k$. We use $\delta-\epsilon$ method. Fix $A \in \mathcal{S}$, for any $\epsilon > 0$, we show there exists $\delta > 0$ such that for all $B \in \mathcal{S}$, $\|A - B\|_{T} < \delta$ implies $\|\tilde{t}(B) - \tilde{t}(A)\|_{\infty} < \epsilon$. It suffices to show the continuity of $\tilde{t}_k$ on $A$ for $k \in [K]$ since we could later pick $\delta := \min_k \delta_k$. Also observe that $\sup_{t \in [0,T]} |A_t^k - B_t^k| \leq \|A-B\|_{T}$, which will be frequently used in context. \\
\noindent\textbf{Case 1 ($A_0^k < c^k$, i.e. $\tilde{t}_k(A) = t_k = 0$). } By continuity of $A^k$ on $t$ and compactness of $[0,T]$, its minimum exists and we let $d_k := \min_{t \in [0,T]}(c^k - A_t^k) > 0$. Let $\delta_k := d_k/2$. If $\|A-B\|_T \leq \delta_k$, then for any $t \in [0,T]$, we have 
$$B_t^k \leq A_t^k+\left|B_t^k-A_t^k\right|<A_t^k+\delta_k \leq c^k-\frac{d_k}{2}<c^k.$$ Therefore we have $\tilde{t}_k(B) = \tilde{t}_k(A) = 0$ and it trivially follows the desired result. \\
\noindent\textbf{Case 2 ($A_0^k > c^k$, i.e. $\tilde{t}_k(A) = t_k \in (0,T)$). } Fix $\epsilon > 0$. Let $l := \max\{0, t_k - \epsilon\}$ and $r := \min\{T, t_k + \epsilon\}$. By negative derivative of $A^k$ at $t_k$, it follows $A_{l}^k > c^k$ and $A_r^k < c^k$. Define $m_k(\epsilon):=\min \left\{A_{l}^k-c^k, c^k-A_r^k\right\}>0$ and pick $\delta_k := \frac{1}{2}m_k(\epsilon)$, we have if $\|A - B\|_T < \delta_k$, then
\begin{equation*}
    \begin{aligned}
        B_{l}^k-c^k &\geq\left(A_{l}^k-c^k\right)-\left|B_{l}^k-A_{l}^k\right|>m_k(\epsilon)-\delta_k=\delta_k>0, \\
        B_r^k-c^k &\leq\left(A_r^k-c^k\right)+\left|B_r^k-A_r^k\right|<-m_k(\epsilon)+\delta_k=-\delta_k<0
    \end{aligned}
\end{equation*}
Therefore we have $B_l^k > c^k$ and $B_r^k < c^k$; by intermediate value theorem, there exists $s \in (l, r)$ that $B_s^k = c^k$. Since $B \in \mathcal{S}$, such $s$ is unique. We have $\tilde{t}_k(B) \in(\ell, r) \subseteq (t_k-\varepsilon, t_k+\epsilon)$ and this implies $|\tilde{t}_k(B) - t_k| < \epsilon$. \\
\noindent\textbf{Case 3 ($A_0^k = c^k$, where $\tilde{t}(A) = t_k = 0$). } Fix $\epsilon > 0$. Under this case, we have $A_t^k<c^k$ for any $t \in(0, T]$.  So by compactness we define $m_k(\epsilon):=\min _{t \in[\epsilon, T]}\left(c^k-A_t^k\right)>0$ and pick $\delta_k := \frac{1}{2}m_k(\epsilon)$. Note that if $\epsilon \geq T$, the meeting time difference is trivially bounded by $\epsilon$ by the form of $\tilde{t}$. Then if $\|A - B\|_T < \delta_k$, it follows for any $t \in [\epsilon, T]$, there is $B_t^k < c^k$. Therefore, if $B_t^k$ meets with $c^k$ at some time, it must happen within $[0,\epsilon)$ and immediately implies $|\tilde{t}_k(B) - t_k| = \tilde{t}_k(B) < \epsilon$; if $B_t^k$ doesn't meet with $c^k$, we also have $|\tilde{t}_k(B) - t_k| = 0 - 0 < \epsilon$. \\
Now we have proved the continuity of $\tilde{t}_k$ on $\mathcal{S}$, taking $\delta := \min_k \delta_k$ completes the proof.

\section*{Proof of Theorem~\ref{thm:aware_bang_bang}}
    We begin by proving the at-most one-jump property. The argument follows a technical path similar to that of Theorem~\ref{thm:charac_bang_bang}, though additional terms introduce extra technical subtleties and require further modifications. We therefore focus  on the parts of the proof that differ from the earlier result. 

    We first define $\delta_k:=c_\rmI^k-\max\{c_p^{k,\rmS}-c_p^{k,\rmI},0\}$. We have $\delta_k \geq c_\rmI^k - c_p^{k,\rmS} > 0$ by the assumption $c_\rmI^k > c_p^{k,\rmS}$. Also choose small enough time horizon $T > 0$ such that 
    $(c_I^k+c_p^{k,I})(1-e^{-\gamma^k T})<\delta_k, \forall k\in[K].$

    Fix $k \in [K]$. Since $u_t^k(\rmR) = 0$ for all $t \in [0,T]$, the state I value function solves
    $\dot u_t^k(\rmI)=\gamma^k u_t^k(\rmI)-c_\rmI^k-c_p^{k,\rmI}P_t(\rmI),
u_T^k(\rmI)=0$.
Using the variation of parameters method, we derive the integral form
$$u_t^k(\rmI)
=
\int_t^T e^{-\gamma^k(s-t)}\big(c_\rmI^k+c_p^{k,\rmI}P_s(\rmI)\big)ds.$$
We thereby obtain
\begin{equation}
    \label{eq:bdd_u_I_delta}
    0\leq \gamma^k u_t^k(\rmI)
\leq
(c_\rmI^k+c_p^{k,\rmI})(1-e^{-\gamma^k(T-t)})
\leq
(c_\rmI^k+c_p^{k,\rmI})(1-e^{-\gamma^kT})
<
\delta_k
\quad \forall t\in[0,T].
\end{equation}
The first inequality follows from the definition of the value function, and the second inequality is by direct integration. We first show that
$u_t^k(\rmI)>u_t^k(\rmS), \forall t\in[0,T)$. Observe that at time $T$, there is $\dot{u}_T^k(\rmS)=-c_p^{k,\rmS}P_T(\rmI)\leq 0$ and $\dot{u}_T^k(\rmI)=-c_\rmI^k-c_p^{k,\rmI}P_T(\rmI) < 0$. Subtract the two terms to have
$$\dot u_T^k(\rmS)-\dot u_T^k(\rmI)
=c_\rmI^k-(c_p^{k,\rmS}-c_p^{k,\rmI})P_T(\rmI)
\geq
c_\rmI^k-\max \{c_p^{k,\rmS}-c_p^{k,\rmI},0\}
=\delta_k>0.$$
Therefore $u^k_{T-}(\rmS)<u^k_{T-}(\rmI)$. By contradiction, assume $u^k_t(\rmS) < u^k_t(\rmI)$ does not hold for some $t \in [0,T)$ and let $\tau \in [0,T)$ be the largest time point at which $u^k_t(\rmS) = u^k_t(\rmI)$. The continuity of the value functions implies \begin{equation}
    \label{eq:us_less_ui_popaware}
    \dot u_\tau^k(\rmS)\le \dot u_\tau^k(\rmI).
\end{equation} On the other hand, it follows
$$\dot u_\tau^k(\rmS)
=(\kappa^k u_\tau^k(\rmS)-c_\nu^k)\hat\nu_\tau^k-c_p^{k,S}P_\tau(\rmI)
\geq
-c_p^{k,\rmS}P_\tau(\rmI) \text{ and } \dot u_\tau^k(\rmI)=\gamma^k u_\tau^k(\rmI)-c_\rmI^k-c_p^{k,\rmI}P_\tau(\rmI).$$
Subtracting the two inequalities gives
$$\dot{u}_\tau^k(\rmS)-\dot{u}_\tau^k(\rmI)
\ge
c_\rmI^k-(c_p^{k,\rmS}-c_p^{k,\rmI})P_\tau(\rmI)-\gamma^k u_\tau^k(\rmI)
\ge
\delta_k-\gamma^k u_\tau^k(\rmI)
>0,$$
where the last inequality follows from~\eqref{eq:bdd_u_I_delta}. This contradicts with~\eqref{eq:us_less_ui_popaware}, we obtain $u_t^k(\rmI)>u_t^k(\rmS), \forall t\in[0,T)$. It remains to show that each time $u_t^k(\rmS)$ meets the threshold $c_\nu^k / \kappa^k$, the crossing is strictly downward, that is, the derivative is negative at the meeting time. This argument is analogous to that in the proof of Theorem~\ref{thm:charac_bang_bang}, and is therefore omitted.
    \par
    We now prove that, for a fixed mean-field environment $Z$, the best response vaccination jump time under population-awareness is no less than the one under baseline formulation. Note that when $Z$ is fixed, the composite infected proportion $P(\rmI)$ is also fixed. We introduce an intensity parameter $\theta \in \mathbb{R}$ and consider the
    awareness cost terms for agent $k$ together as $\theta\big(c_p^{k,\rmS}P_s(\rmI)\mathds{1}_{\{X_s^k=\rmS\}}
+c_p^{k,\rmI}P_s(\rmI)\mathds{1}_{\{X_s^k=\rmI\}}\big)$ at time $s \in [0,T]$. Then consider the corresponding FBODE based on $\theta$, denote the value function that solves it as $u_t^{k,\theta}(e)$, it suffices to show that $u_t^{k,\theta}(\rmS)$ is non-decreasing on $\theta$. The superscript $\theta$ will be similarly notated for other terms throughout the rest of the proof, wherever needed. Fix any admissible control $(\alpha,\nu)$, we denote the total expected cost starting on time $t$ and state $e$
    \begin{equation*}
        \begin{aligned}
            J_{t, e}^{k, \theta}&(\alpha, \nu ; Z):=\mathbb{E}\Big[\int_t^T\Big\{f^k(s, X_s^k, \alpha_s, \nu_s)+\theta(c_p^{k,\rmS}P_s(\rmI)\mathds{1}_{\rmS}(X_s^k)
+c_p^{k,\rmI}P_s(\rmI)\mathds{1}_{\rmI}(X_s^k))\Big\} d s \mid X^k_t=e\Big],
        \end{aligned}
    \end{equation*} where $f^k$ is the running cost for group $k$ in the baseline model. Then for $\theta_2 \geq \theta_1 \geq 0$, there is
    $$\begin{aligned} J_{t, \rmS}^{k, \theta_2}(\alpha, \nu ;Z) &=J_{t, \rmS}^{k, \theta_1}(\alpha, \nu ; Z)+(\theta_2-\theta_1) \EE\Big[
\int_t^T
c_p^{k,\rmS}P_s(\rmI)\mathbf 1_{\rmS}(X_s^k)
+
c_p^{k,\rmI}P_s(\rmI)\mathbf 1_{\rmI}(X_s^k) ds
\mid X_t^k=\rmS
\Big] \\& \geq J_{t, \rmS}^{k, \theta_1}(\alpha, \nu; Z)\end{aligned}$$
    Taking the infimum over all admissible controls yields $u_t^{k,\theta_2}(\rmS)\ge u_t^{k,\theta_1}(\rmS)$ for all $t \in [0,T]$ and $k \in [K]$. This implies that $\hat\nu_t^{k,\theta_2}\geq \hat\nu_t^{k,\theta_1}$ and we thereby have $t_1^{k,\theta_2}\geq t_1^{k,\theta_1}$. In particular, taking $\theta_1=0$ and $\theta_2=1$ shows that the best response
vaccination jump time under population-awareness is no smaller than that under the
baseline model.

\end{document}